\documentclass[12pt,reqno]{amsart}
\usepackage[a4paper,margin=30mm]{geometry}

\usepackage{verbatim,upref,amsxtra,amsthm,xcolor,amssymb,amsxtra}

\usepackage[mathscr]{eucal}
\usepackage{bbm}
\usepackage{dsfont}

\usepackage[initials,nobysame]{amsrefs}

\usepackage{mathtools}
\usepackage[normalem]{ulem}

\usepackage{varioref}

\theoremstyle{plain}
\newtheorem{thm}{Theorem}[section]
\newtheorem{lemma}[thm]{Lemma}
\newtheorem{proposition}[thm]{Proposition}

\theoremstyle{definition}

\theoremstyle{remark}

\newtheorem{example}{\bf Example}

\numberwithin{equation}{section}

\title[Weyl-type theorem on LCA groups and applications]{A Weyl-type theorem for Diophantine approximations driven by
LCA groups and applications}

\author{Aihua FAN}
\address{(A. H. Fan) 
	LAMFA, UMR 7352 CNRS, University of Picardie, 33 rue Saint Leu, 80039 Amiens, France
    and
    Wuhan Institute for Math \& AI, Wuhan University, Wuhan 430072,  China}
\email{ai-hua.fan@u-picardie.fr}

\date{} 

\begin{document}

\maketitle

\begin{abstract} 
 We investigate actions of locally compact Abelian (LCA) groups on the torus $\mathbb{T}^n$, motivated by their close connection with Diophantine approximation. While Kronecker’s theorem yields a classical density result, we prove a stronger equidistribution theorem of Weyl type: every such action admits a decomposition into uniquely ergodic subsystems. The proof of this result is based on a characterization of unique ergodicity for actions of amenable groups on compact metric spaces. As consequences, we establish several foundational results for LCA groups, including the Bohr orthogonality of characters along arbitrary Følner sequences, a Bohr mean formula for almost periodic functions, and a Wiener-type theorem on LCA groups characterizing the discrete part of a Borel probability measure through its Fourier transform. An application to numerical analysis is also discussed.

\end{abstract}

\section{Introduction and main results}

Let $G$ be a locally compact Abelian (LCA) group with  dual group $\widehat{G}$. Let $n \ge 1$ be an integer. 
We fix $n$ elements  $g_1, g_2, \cdots, g_n$ of $G$
and $n$ real numbers $\theta_1, \theta_2, \cdots, \theta_n$ (we can assume that $\theta_j \in [0,1)$ because only  
the fractional part of $\theta_j$ will be concerned).
The question of simultaneous  Diophantine approximation asks: under what condition the following holds: for any $\epsilon>0$ there exists a character 
$\gamma \in \widehat{G}$ such that
\begin{equation}\label{eq:Gdioph}
   \forall j \in \{1, 2, \cdots, n\}, \quad |\gamma(g_j) - e^{2\pi i \theta_j}| <\epsilon. 
\end{equation}

The answer is given in the following theorem. 

\begin{thm}[Kronecker] \label{thm:K}
Let $n \ge 1$ be an integer. Let  $g_1, g_2, \cdots, g_n$ be $n$ elements of
a locally compact Abelian  $G$ and  let $\theta_1, \theta_2, \cdots, \theta_n$ be $n$ real numbers.  A necessary and sufficient condition for 
the system \eqref{eq:Gdioph} to admit  a solution $\gamma \in \widehat{G}$ for every $\epsilon >0$ is
\begin{equation}\label{eq:WC}
   (u_1, \cdots, u_n) \in \mathbb{Z}^n, \ \ \ \sum_{j=1}^n u_j g_j =0\in G \ \ \  \Longrightarrow \ \ \ \sum_{j=1}^n u_j\theta_j  = 0 \pmod 1.
\end{equation}
\end{thm}

 Y.~Meyer presents a proof of Theorem~\ref{thm:K} in \cite[p.~42]{Meyer1972} by means of the Bohr compactification. Chapter~III of Cassels' book \cite{Cassels1957} is devoted to a proof  in the classical setting \(G=\mathbb{T}^n\). We point out, however, that for the simultaneous Diophantine approximation of countably many real numbers, the Kronecker condition \eqref{eq:WC} is no longer sufficient. This motivates the introduction of harmonious sets by Meyer \cite{Meyer1972}; see also \cite{Meyer1970}. Related to this is the theory of model sets, which has been studied in depth and which furnishes mathematical models for quasicrystals in physics; see \cite{BG2013}.

We present an alternative proof of Theorem~\ref{thm:K} by introducing a \(\widehat{G}\)-action on the \(n\)-dimensional torus \(\mathbb{T}^n\) and recasting the underlying problem of Diophantine approximation in the language of dynamical systems. This approach leads naturally to a stronger equidistribution result of Weyl type; see Theorem~\ref{thm:W} below. Our initial motivation for this point of view arises from scientific computing \cite{JLZ2024,JZ2014,JZ2018}, where quasi-periodic functions are  realized as periodic functions on higher-dimensional tori.
We hope that our presentation here is direct and elementary.

To describe our framework and state the main result, it is convenient to identify the torus
\(\mathbb{T}=\mathbb{R}/\mathbb{Z}\) with the multiplicative group
\(\mathbb{S}=\{z\in\mathbb{C}: |z|=1\}\).
For notational simplicity, we continue to write \(\mathbb{T}\) for \(\mathbb{S}\), so that a typical point of \(\mathbb{T}\) is written as \(e^{2\pi i t}\) with \(t\in[0,1)\).
Then every character of \(\mathbb{T}^n\) is of the form
\[
   \chi_{\mathbf{u}}(\mathbf{z})=z_1^{u_1}\cdots z_n^{u_n},
   \qquad
   \mathbf{u}=(u_1,\dots,u_n)\in\mathbb{Z}^n.
\]

For each \(\gamma\in\widehat{G}\), we define a map \(\Phi_\gamma:\mathbb{T}^n\to\mathbb{T}^n\) by
\begin{equation}\label{eq:G^-action}
    \Phi_\gamma(\mathbf{z})
    =
    \bigl(z_1\gamma(g_1),\dots,z_n\gamma(g_n)\bigr),
    \qquad
    \mathbf{z}=(z_1,\dots,z_n)\in\mathbb{T}^n.
\end{equation}
It is the translation led by $(\gamma(g_1), \cdots, \gamma(g_n))$.
It is immediate that
\[
\Phi_{\gamma'\gamma''}=\Phi_{\gamma'}\circ\Phi_{\gamma''},
\]
and that the map
\[
(\gamma,\mathbf{z})\mapsto \Phi_\gamma(\mathbf{z})
\]
is continuous from \(\widehat{G}\times\mathbb{T}^n\) into \(\mathbb{T}^n\).
Thus \((\Phi_\gamma)_{\gamma\in\widehat{G}}\) defines a continuous action of \(\widehat{G}\) on \(\mathbb{T}^n\).

As usual, for \(\mathbf{z}\in\mathbb{T}^n\), we denote its orbit by
\[
{\rm Orb}(\mathbf{z})
:=
\{\Phi_\gamma(\mathbf{z}):\gamma\in\widehat{G}\}.
\]
In particular, if \(\mathbf{1}=(1,\dots,1)\) denotes the identity element of \(\mathbb{T}^n\), then
\({\rm Orb}(\mathbf{1})\) is a subgroup of \(\mathbb{T}^n\), and hence so is its closure
\[
H:=\overline{{\rm Orb}(\mathbf{1})}.
\]
Thus \eqref{eq:Gdioph} is solvable for every $\epsilon>0$ if and only if
\(
(e^{2\pi i\theta_1},\dots,e^{2\pi i\theta_n})\in H.
\)

The preceding Kronecker theorem tells us that \eqref{eq:Gdioph} is solvable if and only if  the Kronecker condition \eqref{eq:WC} is satisfied. We shall strengthen this statement by proving an equidistribution result of Weyl type.
 More precisely, we shall prove that, for every Riemann integrable function \(\varphi\) on \(H\) and every Følner sequence \((F_n)\) in \(\widehat{G}\), one has
\begin{equation}\label{eq:equiv}
    \lim_{n\to \infty} \frac{1}{|F_n|}\int_{F_n}\varphi(\Phi_\gamma(\mathbf{1})) \, d\gamma
    =
    \int_H \varphi(\mathbf{z})\, d\mathbf{z},
\end{equation}
where \(d\mathbf{z}\) denotes the Haar measure on \(H\), \(d\gamma\) the Haar measure on \(\widehat{G}\), and \(|F_n|\) the Haar measure of \(F_n\).
Here, by a Følner sequence, we mean an increasing sequence of compact subsets \((F_n)\) of \(\widehat{G}\) such that, for every \(a\in \widehat{G}\),
\[
\frac{|F_n \Delta (F_n+a)|}{|F_n|}\longrightarrow 0
\qquad\text{as } n\to\infty.
\]
Such sequences exist in every LCA group; see \cite{Greenleaf1973}. We refer to \S\ref{sect:UE} for further discussion of Følner sequences in general amenable locally compact groups.

The equidistribution formula \eqref{eq:equiv} follows from a description of the global dynamical behavior of the \(\widehat{G}\)-action on \(\mathbb{T}^n\) defined by \((\Phi_\gamma)_{\gamma\in\widehat{G}}\). More precisely, we establish the following uniquely ergodic decomposition.

\begin{thm}[Weyl type]\label{thm:W}
Consider the  $\widehat{G}$-action on $\mathbb{T}^n$ defined by $(\Phi_{\gamma})_{\gamma\in \widehat{G}}$ (cf. \eqref{eq:G^-action}).
Let $H=\overline{{\rm Orb}(\mathbf{1})}$.  The group $\mathbb{T}^n$ is decomposed into cosets $\mathbf{z}H$ with $\mathbf{z}\in \mathbb{T}^n$
and the following hold
\begin{itemize}
\item[(a)] The   $\widehat{G}$-action $(\Phi_{\gamma})_{\gamma\in \widehat{G}}$ restricted on each coset $\mathbf{z}H$ is uniquely ergodic. 
\item[(b)] Each coset  $\mathbf{z}H$ is a minimal invariant set.
\item[(c)]
For any continuous function $f: H \to \mathbb{C}$ we have
$$
  \forall z\in H, \quad    \lim_{n\to +\infty} \frac{1}{|F_n|}\int_{F_n} f(\Phi_\gamma(z)) d\gamma  = \int_H f( x) d\mathbf{m}_H(x)
$$
where $\mathbf{m}_H$ is the normalized Haar measure of $H$.
\item[(d)] The  $\widehat{G}$-action on $\mathbb{T}^n$ is uniquely ergodic if and only if the following Kronecker condition is satisfied: 
\begin{equation}\label{eq:UEC}
   (u_1, \cdots, u_n) \in \mathbb{Z}^n, \ \ \ \sum_{j=1}^n u_j g_j =0 \in G\ \ \  \Longrightarrow \ \ \ u_1=\cdots =u_n=0.
\end{equation}

\end{itemize} 
\end{thm}

Let us consider the special case \(G=\mathbb{R}^d\) (\(d\ge 1\)). Then \(\widehat{G}=\mathbb{R}^d\). In this setting, following \cite{Cassels1957}, we consider \(n\) linear forms on \(\mathbb{R}^d\),
\[
L_j(\mathbf{x})=\mathbf{\ell}_j\cdot \mathbf{x},
\qquad
1\le j\le n,
\]
where ``\(\cdot\)'' denotes the Euclidean inner product. Note that each \(\mathbf{\ell}_j\) is an element of the group \(\mathbb{R}^d\). The vectors \(\mathbf{\ell}_1,\dots,\mathbf{\ell}_n\) play the role of the group elements \(g_1,\dots,g_n\) in the definition of the action \eqref{eq:G^-action}. Indeed, every \(\mathbf{x}\in\mathbb{R}^d\) defines a character of \(\mathbb{R}^d\) by
\[
\gamma_{\mathbf{x}}(\mathbf{y})=e^{2\pi i \mathbf{x}\cdot \mathbf{y}}.
\]
Hence
\[
\gamma_{\mathbf{x}}(\mathbf{\ell}_j)=e^{2\pi i \mathbf{\ell}_j\cdot \mathbf{x}},
\]
and these quantities are precisely those appearing in the Diophantine approximation problem; cf.\ \eqref{eq:Gdioph}. See \cite[Chap.~III, Theorem~IV]{Cassels1957} for Kronecker's theorem in this setting.
In this case, the formula \eqref{eq:G^-action} defines an \(\mathbb{R}^d\)-flow on \(\mathbb{T}^n\). By Theorem~\ref{thm:W}, this flow is uniquely ergodic if and only if the vectors \(\mathbf{\ell}_1,\dots,\mathbf{\ell}_n\) are \(\mathbb{Q}\)-linearly independent. This criterion for unique ergodicity was proved by a direct argument in \cite{FJZ2025}. Such unique ergodicity is crucial in the representation of quasiperiodic functions by periodic functions; see \cite{FJZ2025}. We also refer to \cite{JLZ2024} for numerical analysis of quasiperiodic systems based on such representations.

Similarly, when \(G=\mathbb{T}^d\) (\(d\ge 1\)), we have \(\widehat{G}=\mathbb{Z}^d\). In this case, the linear forms
\(
L_j(\mathbf{x})=\mathbf{\ell}_j\cdot \mathbf{x},
\  1\le j\le n,
\)
induce a \(\mathbb{Z}^d\)-action on \(\mathbb{T}^n\). The quantities
\(
\gamma_{\mathbf{x}}(\mathbf{\ell}_j)=e^{2\pi i \mathbf{\ell}_j\cdot \mathbf{x}}
\)
then arise naturally in the corresponding Diophantine approximation problem, with \(\mathbf{x}\in\mathbb{Z}^d\); cf.\ \eqref{eq:Gdioph}. By Theorem~\ref{thm:W}, this \(\mathbb{Z}^d\)-action is uniquely ergodic if and only if, for every \((u_1,\dots,u_n)\in\mathbb{Z}^n\), 
\[
\sum_{j=1}^n u_j \mathbf{\ell}_j =0 \mod \mathbb{Z}^d\quad 
\Longrightarrow \quad
u_1=\cdots=u_n=0.
\]
This condition is stronger than the \(\mathbb{Q}\)-linear independence of the vectors \(\mathbf{\ell}_1,\dots,\mathbf{\ell}_n\). This criterion of unique ergodicity for $\mathbb{Z}^d$-actions  was also established by a direct argument in \cite{FJZ2025}.

In both of the above cases, namely \(\widehat{G}=\mathbb{R}^d\) or \(\widehat{G}=\mathbb{Z}^d\), one may think of \(d\) as the ``dimension of time", or equivalently, as the number of independent time directions. When \(d=1\), the \(\widehat{G}\)-action reduces to a translation action, whose unique ergodicity has been extensively studied; see, for instance, \cite{Furst1981,Glasner2003}.

Now consider the conjugacy question of two minimal $\widehat{G}$-actions.
Let
\[
g=(g_1,\dots,g_n)\in G^n, \qquad h=(h_1,\dots,h_n)\in G^n.
\]
For each \(\gamma\in \widehat G\), consider the following translations of
\(
\mathbb T^n:=(\mathbb S^1)^n
\)
:
\[
\Phi^g_\gamma(z_1,\dots,z_n)
=
\bigl(z_1\gamma(g_1),\dots,z_n\gamma(g_n)\bigr),
\]
and
\[
\Phi^h_\gamma(z_1,\dots,z_n)
=
\bigl(z_1\gamma(h_1),\dots,z_n\gamma(h_n)\bigr).
\]
Recall that the operation in $\mathbb{T}^n$ is multiplicatively  written: $$
z\cdot z' =(z_1z_1', \cdots, z_nz_n')$$
for $z=(z_1, \cdots, z_n) \in \mathbb{T}^n$ and  $z'=(z_1', \cdots, z_n') \in \mathbb{T}^n$.
Also recall that the group of automorphisms of $\mathbb{T}^n$ is isomorphic to the group $GL(n,\mathbb Z)$ consisting of integral matrices with determinant equal to $1$ or $-1$. 

\begin{thm}\label{thm:conj-LCA-torus-actions}
Assume that both \(\widehat G\)-actions \((\Phi^g_\gamma)_{\gamma\in\widehat G}\) and
\((\Phi^h_\gamma)_{\gamma\in\widehat G}\) are minimal on \(\mathbb T^n\). Then the following are equivalent:
\begin{enumerate}
\item The two actions are topologically conjugate, i.e. there exists a homeomorphism
\(
H:\mathbb T^n\to \mathbb T^n
\)
such that
\(
H\circ \Phi^g_\gamma=\Phi^h_\gamma\circ H
\) for all
\( \gamma\in \widehat G.
\)
\item There exists a matrix \(P=(p_{ij})\in GL(n,\mathbb Z)\) such that
\begin{equation}\label{eq:Pg=h}
h_i=\sum_{j=1}^n p_{ij}g_j,
\qquad 1\le i\le n.
\end{equation}
\end{enumerate}
Moreover, every conjugacy \(H\) is affine: there exist \(P\in GL(n,\mathbb Z)\) and
\(c\in \mathbb T^n\) such that
\[
H(z)=c\cdot \Psi_P(z)
\qquad (z\in \mathbb T^n),
\]
where \(\Psi_P:\mathbb T^n\to \mathbb T^n\) is the torus automorphism defined by the matrix $P$:
\begin{equation}\label{eq:autom}
\Psi_P(z_1,\dots,z_n)
=
\left(
\prod_{j=1}^n z_j^{p_{1j}},
\ \dots,\
\prod_{j=1}^n z_j^{p_{nj}}
\right).
\end{equation}
The matrix $P$ verifying \eqref{eq:Pg=h}, namely $h=Pg$,  is unique.
\end{thm}


Now let us present,  as consequence of Theorem~\ref{thm:W},  the following fundamental property common to all locally compact Abelian groups.

\begin{thm}[Bohr orthogonality] \label{thm:group} Let $G$ be a LCA group with its dual group $\widehat{G}$. Let $(F_n)$ be a F{\o}lner sequence in $\widehat{G}$. Then
for any $g\in G\setminus\{0\}$ we have
\begin{equation}\label{eq:W0}
    \lim_{n\to \infty}\frac{1}{|F_n|} \int_{F_n} \gamma(g) d \gamma =0.
\end{equation}
\end{thm}

We refer to this property as the \emph{Bohr orthogonality} of the characters of \(\widehat{G}\), since it immediately implies that for any two distinct points \(x_1,x_2\in G\),
\[
\lim_{n\to\infty}\frac{1}{|F_n|}\int_{F_n}\gamma(x_1)\overline{\gamma(x_2)}\,d\gamma=0.
\]

It is worth pointing out the following corollary: for any Følner sequence \((F_n)\) in \(\mathbb{Z}^d\) and any \(x\in\mathbb{R}^d\) with
\(
x\neq 0 \mod \mathbb{Z}^d,
\)
we have
\[
\lim_{n\to\infty}\frac{1}{|F_n|}\sum_{\mathbf{k}\in F_n}e^{2\pi i \mathbf{k}\cdot x}=0.
\]
For particular Følner sequences such as \(([-n,n]^d\cap\mathbb{Z}^d)\), the sum
\(
\sum_{\mathbf{k}\in F_n}e^{2\pi i \mathbf{k}\cdot x}
\)
can be computed explicitly, and the above limit follows directly. By contrast, for a general Følner sequence, a direct proof by computation doesn't work. 

As a consequence of Theorem~\ref{thm:group}, we then obtain the following Wiener theorem for locally compact Abelian groups. For the classical Wiener theorem on $\mathbb{T}$, see \cite{Katznelson2004}.

\begin{thm}[Wiener] \label{thm:Wiener} Let $G$ be a LCA group with its dual group $\widehat{G}$. Let $(F_n)$ be a F{\o}lner sequence in $\widehat{G}$. 
We consider a regular finite complex measure $\mu$ on $G$.
\begin{enumerate}
\item For any $x\in G$, we have
\begin{equation}\label{eq:W1}
    \mu(\{x\})=
    \lim_{n\to \infty}\frac{1}{|F_n|} \int_{F_n} \widehat{\mu}(\gamma)\gamma(x) d \gamma.
\end{equation}
\item The discrete part of $\mu$ satisfies
\begin{equation}\label{eq:W2}
   \sum_{x\in G} |\mu(\{x\})|^2
   = \ \lim_{n\to \infty}\frac{1}{|F_n|} \int_{F_n} |\widehat{\mu}(\gamma)|^2 d \gamma.
\end{equation}
\end{enumerate}
\end{thm}

It is worth offering some comments now on the above two theorems. 
Schulte \cite{Schulte2021} proved the existence of the limit in \eqref{eq:W0}, but did not show that it vanishes for \(g\neq 0\).
His argument relies on the mean ergodic theorem for amenable groups (see \cite[Corollary~3.4]{Greenleaf1973}), which does not in general imply the vanishing of the limit, since the involved group action  need not be ergodic.
As we shall see, \eqref{eq:W0} can be established by exploiting the unique ergodicity of the \(\widehat{G}\)-action on the closed orbit of \(1\in\mathbb{T}\), where the action is given by
\[
\Phi_\gamma(z)=z\,\gamma(g), \qquad z\in\mathbb{T}.
\]
On the other hand, the formula \eqref{eq:W2} was obtained in \cite{Schulte2021} in a more general framework, which we now recall.

Let \(G\) be equipped with a \(\widehat{G}\)-action, and denote by \(\langle g,\gamma\rangle\) the action of \(\gamma\in\widehat{G}\) on \(g\in G\).
The usual duality pairing \(\gamma(g)\) is a particular example of such an action, and it is precisely the one considered in Theorem~\ref{thm:group}.
Now let \(\nu=(\nu_N)_{N\ge 1}\) be a sequence of probability measures on \(\widehat{G}\), called an \emph{averaging sequence}.
For \(g\in G\), we define
\[
c_\nu(g)=\lim_{N\to\infty}\int_{\widehat{G}}\langle g,\gamma\rangle\,d\nu_N(\gamma),
\]
whenever the limit exists.
If this limit exists for every \(g\in G\), the averaging sequence \(\nu\) is said to be \emph{good}.
If, in addition, \(c_\nu(g)=0\) for every nontrivial \(g\in G\) (while \(c_\nu(e_G)=1\) for the identity element \(e_G\in G\)), then \(\nu\) is said to be \emph{ergodic}.

In this terminology, Theorem~\ref{thm:group} asserts that, for the standard action of \(\widehat{G}\) on \(G\), the averaging sequence \((\nu_{F_n})\) associated with any Følner sequence \((F_n)\) is ergodic.
Here \(\nu_{F_n}\) denotes the normalized Haar measure on \(F_n\), namely
\[
\nu_{F_n}(B)=\frac{{\rm meas}(B\cap F_n)}{{\rm meas}(F_n)},
\]
where \({\rm meas}(\cdot)\) denotes the Haar measure on \(\widehat{G}\).
The formula \eqref{eq:W2} was proved in \cite{Schulte2021} for general ergodic averaging sequences.
We would like to point out that \eqref{eq:W1} also remains valid for every ergodic averaging sequence.
Moreover, the proofs of both \eqref{eq:W1} and \eqref{eq:W2} in that general setting are essentially identical to the argument for \((\nu_{F_n})\), which will be given later in \S\ref{sect:group}.
For the sake of simplicity, however, we state Theorem~\ref{thm:Wiener} only for the particular sequence \((\nu_{F_n})\).

The space of Bohr almost periodic functions admits an invariant mean, called Bohr mean; see \cite{Besicovitch1954,Bohr1947,Corduneanu1989,LZ1982,Neumann1934}.
Using Theorem~\ref{thm:group}, together with the fact that almost periodic functions are uniform limits of trigonometric polynomials
(the Bohr approximation theorem, which remains valid for all LCA groups; see \cite{Neumann1934}),
we obtain the following formula for the Bohr mean.

Let \(AP(G)\) denote the space of all complex-valued almost periodic functions on \(G\) in the sense of Bohr. 
It is well known that a continuous translation-invariant linear functional
\(\mathcal{M}:AP(G)\to\mathbb{C}\) such that $\mathcal{M}(1) =1$, and that it is uniquely characterized by these properties.
The existence of Bohr mean is a fundamental result in  Bohr's theory of almost periodic functions. 
The following theorem provides another proof of this fundamental result by defining the Bohr mean using directly F{\o}lner sequences.

\begin{thm}[Mean formula]
\label{thm:BM}
    For any almost periodic function $f\in AP(G)$, the following hold:
    \begin{itemize} 
    \item [{\rm (i)}\ ] \ The following limit exists and is independent of $g$:
     \begin{equation}\label{eq:MU}
\forall g\in G, \quad \mathcal{M}(f) = \lim_{n\to \infty}\frac{1}{|F_n|}\int_{F_n}f(x+g) dx
\end{equation}
where  $\{F_n\}$ is any given F{\o}lner sequence. 
 \item  [{\rm (ii)}\ ]  The limit in \eqref{eq:MU}  is independent of the choice of F{\o}lner sequence.
\item  [{\rm (iii)}]  The limit in \eqref{eq:MU} is uniform in $g\in G$.
\end{itemize}
\end{thm}

The formula \eqref{eq:MU} holds for weakly almost periodic function on discrete amenable groups (see \cite{KL2016},  Proposition D.17 in Appendix D).

A key ingredient in the proof of our main result, Theorem~\ref{thm:W}, is a characterization of unique ergodicity for actions of amenable groups on compact metric spaces; this will be developed in Section~\ref{sect:UE} (see Theorem \ref{thm:UE}). Theorem~\ref{thm:K},  Theorem~\ref{thm:W}  and Theorem~\ref{thm:conj-LCA-torus-actions} concerning the $\widehat{G}$-actions are proved in Section~\ref{sect:KW} where some concrete examples are examined. Section~\ref{sect:group} is devoted to applications of Theorem~\ref{thm:W} to harmonic analysis on locally compact Abelian groups, including the proofs of the Bohr orthogonality of group characters (Theorem~\ref{thm:group}), the generalized Wiener theorem (Theorem~\ref{thm:Wiener}), and the formula for the Bohr mean (Theorem~\ref{thm:BM}). The last section is devoted to an application to  the numerical solution of eigenproblem of Schr\"odinger equation.
\medskip

{\em Acknowledgement.}  The author thanks Kai Jiang, Hanfeng Li, Yves Meyer, Herv\'e Queff\'elec, Ruxi Shi and Benjamin Weiss for their interests, questions and informations. 
This work is supported by the NSFC (No. 12231013, No. 12571205).

\section{Unique ergodicity of the action of amenable groups}
\label{sect:UE}

Let \(G\) be a  locally compact group with Haar measure \(m_G\), not necessarily abelian.
A continuous action of \(G\) on a compact metric space \(X\) is given by a homomorphism
\(
\Phi:G\to {\rm Homeo}(X),
\)
where \({\rm Homeo}(X)\) denotes the group of homeomorphisms of \(X\). In other words,
\[
\Phi_{g'}\circ \Phi_g=\Phi_{g'g}
\qquad
\text{for all } g,g'\in G.
\]
We shall write \(g\cdot x\) for \(\Phi_g(x)\), so that
\[
g'\cdot (g\cdot x)=(g'g)\cdot x.
\]
The action is said to be continuous if the map
\(
(g,x)\mapsto g\cdot x
\)
is continuous on $G\times X$.

Let \(C(X)\) denote the Banach space of continuous complex-valued functions on \(X\), equipped with the supremum norm \(\|\cdot\|_\infty\). By the Riesz representation theorem, the dual space \(C(X)^*\) may be identified with the space \(M(X)\) of finite Borel measures on \(X\). We denote by \(M_1^+(X)\) the convex set of Borel probability measures on \(X\), and by \(M_G(X)\) the set of \(G\)-invariant probability measures. A measure \(\mu\in M_1^+(X)\) is said to be \(G\)-invariant if
\[
(\Phi_g)_*\mu=\mu
\qquad
\text{for all } g\in G,
\]
where \((\Phi_g)_*\mu\) denotes the pushforward of \(\mu\) under \(\Phi_g\).

A continuous action of a group \(G\) on a compact space need not admit any invariant probability measure; see \cite{EW2011}.
Such a classical example is the action of the discrete group \({\rm SL}_2(\mathbb{Z})\) on the projective line \(\mathbb{P}^1(\mathbb{R})\); see again \cite{EW2011}. This action is given by
\[
\begin{pmatrix}
a & b\\
c & d
\end{pmatrix}
:[x,y]\mapsto [ax+by,cx+dy].
\]
By contrast, every continuous action of an amenable group on a compact space admits an invariant probability measure.

Recall that a  locally compact group \(G\) is said to be amenable if, for every compact set \(K\subset G\) and every \(\varepsilon>0\), there exists a measurable set \(F\subset G\) with compact closure such that \(KF\) is measurable and
\[
m_G(F\Delta KF)\le \varepsilon\, m_G(F),
\]
where
\[
KF=\{kf:k\in K,\ f\in F\}.
\]
Such a set \(F\) is called \((K,\varepsilon)\)-invariant. When \(\varepsilon=0\), this reduces to the usual notion of \(K\)-invariance.

A sequence \((F_n)_{n\ge 1}\) of compact subsets of \(G\) is called a {\em Følner sequence} if, for every compact set \(K\subset G\) and every \(\varepsilon>0\), all but finitely many \(F_n\) are \((K,\varepsilon)\)-invariant. Associated with such a sequence, one defines the averaging operators \(A_n\) on \(C(X)\) by
\begin{equation}\label{eq:An}
A_n\phi(x)=\frac{1}{m_G(F_n)}\int_{F_n}\phi(g\cdot x)\,dm_G(g),
\qquad \phi\in C(X).
\end{equation}
Amenability of \(G\) is equivalent to the existence of a Følner sequence. In particular, every locally compact Abelian group is amenable; see \cite[p.~113]{Pier1984}. We refer to \cite{EW2011,Glasner2003} for further background on actions of amenable groups.

We now characterize those actions of amenable groups that admit a unique invariant probability measure, that is, the uniquely ergodic actions.
For each \(g\in G\), define the operator \(T_g:C(X)\to C(X)\) by
\[
T_g\phi(x)=\phi(g\cdot x).
\]
This is the usual Koopman operator associated with the action. We then introduce
\begin{equation}\label{eq:IB}
\mathcal{I}:=\bigcap_{g\in G} {\rm Ker} (T_g-I),
\qquad
\mathcal{B}:=\bigcup_{g\in G}{\rm Im} (T_g-I),
\end{equation}
where \({\rm Ker} (T_g-I)\) and \({\rm Im} (T_g-I)\) denote, respectively, the kernel and the image of the operator \(T_g-I\), with \(I\) denoting the identity operator on \(C(X)\).
Clearly, \(\mathcal{I}\) is a closed subspace of \(C(X)\), consisting precisely of the \(G\)-invariant functions. By contrast, \(\mathcal{B}\) is in general not a subspace; we therefore write \({\rm Span}(\mathcal{B})\) for the linear subspace generated by \(\mathcal{B}\).

We now state a characterization of uniquely ergodic \(G\)-actions.

\begin{thm}
\label{thm:UE}
	Let $G$ be a  $\sigma$-compact and locally compact group action on a compact metric space $X$. Suppose that $G$ is amenable and 
	$(F_n)_{n\ge 1}$ is a F{\o}lner sequence.
	Then  the following are
	equivalent:
    \begin{enumerate}
    	\item for all $ \phi \in C(X)$ and all $x \in X$, 
    	$\lim_{n\to \infty} A_n \phi(x) $ exists and is independent of $x$;
    	\item the above limit is uniform in $x$ for every $\phi \in C(X)$;
    	\item we have the direct sum decomposition $C(X) = \mathbb{C} \oplus \overline{{\rm Span}(\mathcal{B})}$;
	\item  the action of $G$ on $X$ is uniquely ergodic;
 \item the conclusion (1) holds for Riemann integrable functions.
    	\end{enumerate}
\end{thm}

The proof of this theorem follows the same general line as the classical proof for \(\mathbb{Z}\)-actions; cf. \cite{Parry1981}. For completeness, we include the details.
Here, by a Riemann integrable function we mean a uniform limit of finite sums of the form
\[
\sum_k a_k\,1_{A_k},
\]
where each set \(A_k\) has boundary of \(\mu\)-measure zero; such sets are referred to as \(\mu\)-continuity sets in \cite[p.~15]{Billingsley1999}.

\subsection{Existence of the limit of averages}
Theorem \ref{thm:UE} will be a consequence of the following more general result, concerning the existence of the limit 
$\lim A_n\phi(x)$ for all continuous functions $\phi$.
Recall that $\mathcal{I}$ and $\mathcal{B}$ appearing in the following theorem are defined by \eqref{eq:IB}.
A proof of Theorem \ref{thm:UE0}, similar to the classical case (cf. \cite{Parry1981}, p.12), is presented for completeness.

\begin{thm}\label{thm:UE0}
	Let $G$ be a  $\sigma$-compact and locally compact group action on a compact metric space $X$. Suppose that $G$ is amenable and 
	$(F_n)_{n\ge 1}$ is a Folner sequence.
	Then  the following are
	equivalent:
    \begin{enumerate}
    	\item for all $ \phi \in C(X)$ and all $x \in X$, the limit exists:
    	$\lim_{n\to \infty} A_n\phi(x)   = \tilde{\phi}(x) \in C(X)$;
    	\item the above limit is uniform in $x$ for every $\phi\in C(X)$;
    	\item we have the direct sum decomposition $C(X) = \mathcal{I} \oplus \overline{{\rm Span}(\mathcal{B})}$.
    	\end{enumerate}
\end{thm}

\begin{proof} The proof is composed of the following steps.

{\sc Step 1.}  {\it For $\phi\in \mathcal{I}$, the uniform convergence is obvious and $\tilde{\phi} =\phi$}.

{\sc Step 2.}  {\it For $\phi\in \overline{\mathcal{B}}$, the uniform convergence also holds and the limit is zero}. Indeed, first assume  $\phi \in B$. Then $\phi=T_{h}\psi-\psi$ for some $h\in G$ and some $\psi\in C(X)$.
From 
$$
   A_n \phi(x) = \frac{1}{m_G(F_n)} \left( \int_{h F_n} \psi(g\cdot  x) dm_G(x) -  \int_{F_n} \psi(g\cdot x) dm_G(x) \right)
$$
we get 
$$
    |A_n\phi(x)| \le \|\psi\|_\infty \frac{m_G(hF_n \Delta F_n)}{m_G(F_n)},
$$
which implies $\|A_n \phi\|_\infty \to 0$. By the linearity, we also have $\|A_n \phi\|_\infty \to 0$ for all $\phi\in {\rm Span}(\mathcal{B})$.
Now assume $\phi \in \overline{ {\rm Span}(\mathcal{B})}$. We approximate $\phi$ by elements $\phi'$ in  ${\rm Span}(\mathcal{B})$, with $\|\phi-\phi'\|\le \epsilon$ for any given $\epsilon >0$.
Since $A_n$ are contractions on $C(X)$, we have
$$
     \|A_n \phi\|_\infty  \le  \|A_n (\phi-\phi')\|_\infty +  \|A_n \phi'\|_\infty \le \epsilon +  \|A_n \phi'\|_\infty.
$$
Now we can conclude for {\sc Step 2}.


{\sc Step 3.}  {\it  From Steps 1 and 2, we get immediately $(3) \Rightarrow (2) \Rightarrow (1) $. It remains to prove $(1) \Rightarrow (3)$}. 
Assume (1). We define  $P: C(X) \to C(X)$ by $
     P(\phi) : = \tilde{\phi}.
$
Clearly, $P$ is linear, continuous and  projective i.e. $P^2 =P$. Any $\phi \in C(X)$ can be decomposed as follows: 
$$
     \phi = P\phi + (\phi - P\phi) \qquad {\rm with}  \   P\phi \in I, \  \phi-P\phi \in {\rm Ker}P.
$$ 
 Thus, we have checked that 
$C(X) = \mathcal{I} + {\rm Ker} P$. It is even a direct sum, because $\phi \in \mathcal{I} \cap {\rm Ker} P$ implies 
$\phi = P\phi = 0$. To finish the proof,  we have only to show
$$ {\rm Ker} P = \overline{{\rm Span}(\mathcal{B})}.$$
By Step 2,  ${\rm Span (\mathcal{B})} \subset  {\rm Ker} P$.
As ${\rm Ker} P$ is closed, we have 
$ \overline{{\rm Span}(\mathcal{B})} \subset {\rm Ker} P$.  
To prove the inverse inclusion $ {\rm Ker} P \subset \overline{\rm Span (\mathcal{B})}$,
by Hahn-Banach theorem, we have only to show the implication
$$
    \forall \mu \in \mathcal{M}(X), \mu|_{\overline{\rm Span (\mathcal{B})}} =0 \Rightarrow \mu|_{{\rm Ker} P} =0.
$$
Actually, we can prove
$$
\forall \mu \in \mathcal{M}(X), 
\forall \phi \in C(X), \forall g\in G, \langle \mu, T_g\phi-\phi \rangle =0
 \Rightarrow \forall \psi \in {\rm Ker} P, \langle \mu, \psi \rangle =0.
$$
Notice that $\langle \mu, T_g \phi-\phi \rangle =0$ ($\forall \phi \in C(T), \forall g\in G$)
means that the measure $\mu$ is $G$-invariant. On the other hand,
$\Phi \in {\rm Ker} P$ means
$$
    \forall x\in X, \ \ \  A_n\psi(x):=  \frac{1}{m_g(F_n)}\int_{F_n} \psi(g\cdot x) dm_G(g) \to 0.
    $$ 
Using Fubini's theorem and the invariance  $\langle \mu, \psi \circ T_g\rangle =  \langle \mu, \psi\rangle$, for all $n\ge 1$ we get
$$
    \langle \mu, \psi\rangle = 
    \frac{1}{m_G(F_n)}   \int_X \left( \int_{F_n} \psi(g\cdot x) dm_G(g) \right)d\mu(x) =  \langle \mu, A_n \psi\rangle.
$$
We conclude  the $ \langle \mu, \psi\rangle =0$, by using Lebesgue's dominated convergence theorem  and the fact $A_n \psi(x) \to 0$.

\end{proof}

\subsection{Unique ergodicity: Proof of Theorem \ref{thm:UE}}  
Having proved Theorem \ref{thm:UE0}, for the equivalence of (1)-(4), we have only two points to prove:
\vspace{0.01em}

{\it A. The unique ergodicity implies that $\lim A_n \phi(x)$ exists and is independent of $x$}.

We first claim that the limit $\lim A_n \phi(x)$ exists for all $\phi\in C(X)$ and all $x\in X$. Otherwise, there exist a function 
$\phi \in C(X)$, a point $z_0\in X$ and integers $n_j', n_j''$ such that 
$$
     A_{n_j'}\phi(z_0) \to a, \quad A_{n_j'}\phi(z_0) \to b \qquad {\rm with } \ \ a\not=b. 
$$
Now consider the probability measures 
$$
   \eta'_j = \frac{1}{m_G(F_{n_j'})} \int_F \delta_{\Phi_{x}(z_0)}dm_G(x), \qquad  \eta''_j = \frac{1}{m_G(F_{n_j''})} \int_F \delta_{\Phi_{x}(z_0)} dm_G(x).
$$
Assume, without loss of generality, that 
$$
      \eta_j' \to \eta', \qquad \eta_j'' \to \eta''.
$$
Both $\eta'$ and $\eta''$ are $\mathbb{R}^d$-invariant, but distinct because
$$
      \int \phi d\eta' =\lim \int \phi d\eta_{j}'= \lim A_{n_j'}\phi(z_0) =a \not=b=   \lim A_{n_j''}\phi(z_0)  =  \lim \int \phi d\eta_{j}''= \int \phi d\eta''.
$$
In a  similar way, we can prove that $\lim A_n \phi(x)$ is constant. 
\medskip

{\it B. That $\lim A_n \phi(x)$ is constant for every $\phi\in C(X)$ implies the unique ergodicity.}

Suppose that $\mu_1$ and $\mu_2$ are two distinct invariant measure. So, there exists a function $\phi \in C(X)$
such that 
\begin{equation}\label{eq:m1m2}\langle \mu_1, \phi\rangle \not= \langle \mu_2, \phi\rangle.
\end{equation}  By the hypothesis, for some constant $c_\phi$
   $$
    \forall z \in X, \quad \lim_{n\to \infty}A_n \phi(z)=c_f.
    $$
    Integrating with respect to $\mu_1$ which is invariant, we get $c_\phi= \langle \mu_1, \phi\rangle$. Similarly, 
 we also have  $c_\phi = \langle \mu_2, \phi\rangle$. These contradict \eqref{eq:m1m2}.

For the equivalence of (1) and (5), it suffices to call for Portmanteau theorem (cf. \cite{Billingsley1999}, p.16).

 \section{Theorems of Kronecker and Weyl,  Theorem of conjugacy
}\label{sect:KW}

We are going to prove Theorem \ref{thm:K},    Theorem \ref{thm:W} and Theorem \ref{thm:conj-LCA-torus-actions}. 
Some examples are examined later.

 \subsection{Proof of Theorem \ref{thm:K} }
 Recall that the annihilator  of a closed subgroup $H$ of $\mathbb{T}^n$ is defined to be the set $H^\perp$ consisting of all characters $\chi_\mathbf{u}$ of $\mathbb{T}^n$, represented by   $u \in \mathbb{Z}^n$,  such that 
$\chi_\mathbf{u}(\mathbf{h})=1$ for every $\mathbf{h}\in H$. Now assume $H=\overline{{\rm Orb}(\mathbf{1})}$,  the closure of the orbit 
\({\rm Orb}(\mathbf{1})\) of  the identity element \(\mathbf{1}=(1,\dots,1)\)  of \(\mathbb{T}^n\). As we have observed that  the inequality \eqref{eq:Gdioph} is solvable if and only if
\(
(e^{2\pi i\theta_1},\dots,e^{2\pi i\theta_n})\in H.
\)
To characterize the elements of $H$, we need to know the  annihilator  of $H$, which is given in the next lemma.

\begin{lemma} \label{lem:Hperp} $H^\perp =\{\mathbf{u}=(u_1, \cdots, u_n)\in \mathbb{Z}^n: u_1 g_1+\cdots+u_n g_n =0 \}$.
\end{lemma}

\noindent {\em Proof of Lemma \ref{lem:Hperp}.}  An element $\mathbf{u}=(u_1, \cdots, u_n)$ of $ \mathbb{Z}^n$ belongs to $H^\perp$ iff $\chi_\mathbf{u}(\mathbf{h}) =1$ for all $\mathbf{h}\in {\rm Orb}(\mathbf{1})$. As a typical point $\mathbf{h}$ in ${\rm Orb}(\mathbf{1})$ has the form 
$\Phi_\gamma(\mathbf{1}) = \gamma(g_1) \cdots \gamma(g_n)$, $\mathbf{u}\in H^\perp$ means 
$$
    \forall \gamma \in \widehat{G}, \quad       \gamma(g_1)^{u_1} \cdots \gamma(g_n)^{u_1}=1, \quad i.e. \quad \gamma (u_1g_1+ \cdots +u_ng_n)=1.
$$
As characters of $G$ separate points of $G$, this means $u_1g_1+ \cdots +u_ng_n$ is the unit of $G$.

 The system \eqref{eq:Gdioph} has a solution $\gamma \in \widehat{G}$ for every positive number $\epsilon>0$ just mean that 
$\mathbf{z}:=(e^{2\pi i\theta_1}, \cdots, e^{2\pi i \theta_n})$ belongs to the closed orbit $H$, the closure of ${\rm Orb}(\mathbf{1})$. If we consider the quotient group
$\mathbb{T}^n/H$, $\mathbf{z}\in H$ means that $\mathbf{z}$ represent the unit element of the quotient $\mathbb{T}^n/H$, which is characterized by
$$
\forall \gamma \in \widehat{\mathbb{T}^n/H}, \quad \gamma(\mathbf{z})=1.
$$ However,  it is well known that $\widehat{\mathbb{T}^n/H} = H^\perp$. So, by Lemma 
\ref{lem:Hperp},  $\mathbf{z} \in H$ iff for all $(u_1, \cdots, u_n)\in \mathbb{Z}^d$ such that $u_1g_1+\cdots u_n g_n=0$ we have $\chi_{\mathbf{u}}(\mathbf{z})=1$, namely
$$
      (e^{2\pi i \theta_1})^{u_1}\cdots  (e^{2\pi i \theta_n})^{u_n} =1, \ \ \ i.e. \ \ \  e^{2\pi i (u_1\theta_1 + \cdots + u_n \theta_n)}=1,
$$
in other words, $u_1\theta_1 + \cdots + u_n \theta_n$ is an integer.

 \subsection{Proof of Theorem \ref{thm:W} }
 
 The proof of Theorem \ref{thm:W} is essentially based on Theorem \ref{thm:UE} and the following lemma which compute the Fourier coefficients of the image measure under 
 $\Phi_\gamma$.
 
\begin{lemma} \label{lem:Fourier} For a  probability measure $\mu$  on $H$, the image measure  $\mu\circ \Phi_\gamma^{-1}$ has its Fourier coefficients equal to
$$
  \widehat{\mu\circ \Phi_\gamma^{-1}} (\chi_\mathbf{u}) 
  =   \gamma(u_1g_1+\cdots +u_n g_n) 
  \widehat{\mu}(\chi_\mathbf{u}).
$$
\end{lemma}

{\em Proof of Lemma \ref{lem:Fourier}.} 
Let us compute the Fourier coefficient of 
$\mu \circ \Phi_\gamma^{-1} $.  Take a typical character $\chi_\mathbf{u}(\mathbf{z})= z_1^{u_1}\cdots z_n^{u_n}$, we have
 \begin{eqnarray*}
     \widehat{\mu\circ \Phi_\gamma^{-1}} (\chi_\mathbf{u})
     &=& \int_H \chi_\mathbf{u} \circ \Phi_\gamma(\mathbf{z}) d\mu(\mathbf{z}).
 \end{eqnarray*}
 We conclude by using   the relation
 $$
    \chi_\mathbf{u} \circ \Phi_\gamma(\mathbf{z}) = (z_1\gamma(g_1))^{u_1} \cdots  (z_n\gamma(g_1))^{u_n} =  \gamma(u_1g_1+\cdots +u_n g_n)  \chi_\mathbf{u}(\mathbf{z}).
 $$

 {\em Proof of (a): Unique ergodicity of the action on $H$.}
 As a closed orbit, $H$ is invariant. We first prove that the $\widehat{G}$-action  $\{\Phi_{\gamma}\}_{\gamma\in \widehat{G}}$ restricted on the subgroup $H$ 
is uniquely ergodic. 
Let $\mu$ be an invariant probability measure on $H$, i.e. $\mu \circ \Phi_\gamma^{-1} =\mu$ for all $\gamma\in \widehat{G}$. 
 We claim that $\mu$ is the Haar measure on $H$, by showing that 
 $$
 \widehat{\mu}(\chi)=0 \ \ \ {\rm  for \ all}\ \ \  \chi \in \widehat{H}\setminus \{1\}.
 $$  
 Recall the well known fact $ \widehat{H}= \mathbb{Z}^n/H^\perp$. But $H^\perp =\{(u_1, \cdots, u_n)\in \mathbb{Z}^n : u_1g_1+\cdots+u_ng_n =0\}$, by Lemma \ref{lem:Hperp}.
 Thus,  a typical character $\chi \in \widehat{H}\setminus \{1\}$ takes the form 
 $\chi(z)= z_1^{u_1}\cdots z_n^{u_n}$ for some $(u_1, \cdots, u_n)\in \mathbb{Z}^n$ such that 
 $u_1g_1+\cdots +u_ng_n \not=0$.  By Lemma \ref{lem:Fourier}, the invariance of $\mu$ means
$$
  \forall \gamma \in \widehat{G}, \quad \gamma(u_1g_1+\cdots +u_n g_n) \widehat{\mu}(\chi) = \widehat{\mu}(\chi).
$$
However, $u_1g_1+\cdots +u_ng_n \not=0$ implies that $\gamma(u_1g_1+\cdots +u_n g_n) \not=1$ for some $\gamma \in \widehat{G}$. 
It follows that $ \widehat{\mu}(\chi)=0$.

{\em Proof of (a): Second argument for the unique ergodicity (by Hanfeng Li)}. Let $\mu$ be a Borel probability measure on $H$ invariant under the $\widehat{G}$-action $(\Phi_\gamma)_{\gamma\in \widehat{G}}$. That is, $\mu$ is invariant under the map $H\rightarrow H$ sending ${\bf x}$ to ${\bf x}{\bf z}'$ for every ${\bf z}'\in {\rm Orb}({\bf 1})$. By continuity, $\mu$ is also invariant under the map $H\rightarrow H$ sending ${\bf x}$ to ${\bf x}{\bf z}$ for every ${\bf z}\in H$. Thus $\mu$ is the Haar measure of $H$.

{\em Proof of (b): Minimality of $H$.}
Take a F{\o}lner  sequence $(F_n)$ of $\widehat{G}$. By Theorem \ref{thm:UE}, the unique ergodicity implies that the average $A_n\varphi(\mathbf{z})$ tends to $\int_H \varphi d\varphi d\mu_H$ for any $\mathbf{z} \in H$
and any Riemann integrable function $\varphi$ defined on  $H$, where 
$\mu_H$ is the Haar measure of $H$. Applying this to indicator functions of balls in $H$, we get that the orbit of any $\mathbf{z}$ in $H$ is dense in $H$.

{\em Proof of (b): Second  argument of the minimality (by Hanfeng Li)}    Let ${\bf z}'\in H$. Then 
$$\overline{{\rm Orb}({\bf z}')}=\overline{{\bf z}'{\rm Orb}({\bf 1})}={\bf z}'H=H.$$
Thus $H$ is minimal invariant.

{\em Proof of (c).}  The limit in (c) is a direct consequence of the unique ergodicity of the action on $zH$ and Theorem \ref{thm:UE}.

{\em Proof of (d).}  According to the decomposition, the action is unique ergodic if and only if  $H=G$, which by Theorem \ref{thm:W} is equivalent to that \eqref{eq:WC}
holds for all $(\theta_1, \cdots, \theta_n)$. The last assertion is equivalent to the Kronecker condition \eqref{eq:UEC}.

{\em Finishing of Proof: Conjugacy of two subsystems $H$ and $zH$.}
We have  proved  (a) and (b) for $H$. This implies that the conclusions of (a) and (b) also hold for an arbitrary coset $\mathbf{z}H$. Indeed, 
let $\tau_\mathbf{z}(\mathbf{w})=\mathbf{z}\mathbf{w}=(z_1w_1, \cdots, z_nw_n)$ be the translation by $\mathbf{z}$ on $\mathbb{T}^n$. We have $$\tau(H)= \mathbf{z}H; \quad \forall \gamma \in \widehat{G}, \ \ \tau_\mathbf{z} \circ \Phi_{\gamma} = \Phi_{\gamma} \circ \tau_\mathbf{z}.$$
So, the action on $\mathbf{z}H$ is conjugate to the action on $H$.

{\em Remark.} \ The proof of Theorem~\ref{thm:W} that we present above is essentially based on Theorem~\ref{thm:UE}, the criterion of unique ergodicity of amenable group actions.  Another conceptual argument is also possible: the
$\widehat{G}$-action $(\Phi_\gamma)_{\gamma\in \widehat{G}}$ defined by \eqref{eq:G^-action} is equicontinuous, even isometric. Therefore, by the theory of equicontinuous actions,  it admits a unique minimal and ergodic decomposition. It is then natural for  F{\o}lner averages come to play: the unique ergodicity of the subsystem $(H, (\Phi_\gamma))$ is equivalent to that the  F{\o}lner averages converge everywhere on $H$ to a constant, which is confirmed by Theorem~\ref{thm:UE}.

\subsection{When $\widehat{G}$ is compact.}
Here are some words about the case where $\widehat{G}$ is compact, like $G=\mathbb{Z}^d$ with $\widehat{G}=\mathbb{T}^n$.  
Then  it is about the action of a compact Abelian group on $\mathbb{T}^d$. This is to answer a question of H. Queff\'elec (personal communication). 

Let us repeat the setting. Take $n$ elements $m_1, \cdots, m_n$ from the discrete group $G$ (these are $n$ integral  lattice points in $\mathbb{Z}^d$ while  $\widehat{G} =\mathbb{T}^d$).
For $\gamma\in \widehat{G}$, we have
$$
\Phi_\gamma (\mathbf{z}) =(z_1 \gamma(m_1), \cdots, z_n \gamma(m_m)).
$$
As $\widehat{G}$ is compact, we take the trivial F$\o$lner sequence $\{F_n\}$ with $F_n =\widehat{G}$ for all $n\ge 1$. 
According to Theorem \ref{thm:K} and Theorem \ref{thm:UE},
the unique ergodicity of $\{\Phi_{\gamma}\}_{\gamma \in \widehat{G}}$ restricted to $H=\overline{{\rm Orb}(\mathbf{1})}$
is described by \eqref{eq:equiv}, which becomes now
\begin{equation}\label{eq:equiv2}
    \int_{\widehat{G}}\varphi(\Phi_\gamma(\mathbf{1})) d\gamma = \int_H \varphi(\mathbf{z})d\mathbf{z}.
\end{equation}
Let us prove this equality directly. The mapping $\gamma \mapsto \Phi_\gamma(\mathbf{1})$ is a group morphism from the compact group 
$\widehat{G}$ into the compact group $H (\subset \mathbb{T}^n)$. By the definition of $H$, this mapping is actually a surjective endomorphism. 
So,  under this mapping, the Haar measure of $H$ is the image of the Haar measure of $\widehat{G}$.  In other words, the equality \eqref{eq:equiv2} 
holds.  In the special case of $\widehat{G}=\mathbb{T}^d$, \eqref{eq:equiv2} reads as follows
$$
    \int_{\mathbb{T}^d} \varphi(e^{2\pi i m_1\cdot t}, \cdots, e^{2\pi i m_n \cdot t}) dt = \int_H \varphi(\mathbf{z})d\mathbf{z}
$$
where $H=\{e^{2\pi i m_1\cdot t}, \cdots, e^{2\pi i m_n \cdot t}: t\in \mathbb{T}^d\}$ is a closed subgroup of $\mathbb{T}^n$.

\subsection{Proof of Theorem \ref{thm:conj-LCA-torus-actions}}
We prove two implications and the uniqueness.

 {\em (2) implies (1)}.
Assume that there exists \(P=(p_{ij})\in GL(n,\mathbb Z)\) such that
\eqref{eq:Pg=h} holds. Since \(P=(p_{ij})\in GL(n,\mathbb Z)\),  the map
defined by \eqref{eq:autom}
is a topological group automorphism of \(\mathbb T^n\). 
We check that $\Psi_P$ is a conjugacy between the two actions $\Phi^h$ and $\Phi^g$. Indeed, for any \(\gamma\in \widehat G\),
as $\Phi_\gamma(z) = z \cdot \Phi_\gamma(1, \cdots,1)$ 
we have
\[
\Psi_P\bigl(\Phi^g_\gamma(z)\bigr)
=
\Psi_P(z)\cdot \Psi_P\bigl(\Phi^g_\gamma(1,\dots,1)\bigr).
\]
Now
\(
\Phi^g_\gamma(1,\dots,1)=\bigl(\gamma(g_1),\dots,\gamma(g_n)\bigr),
\)
hence
\[
\Psi_P\bigl(\Phi^g_\gamma(1,\dots,1)\bigr)
=
\left(
\prod_{j=1}^n \gamma(g_j)^{p_{1j}},
\ \dots,\
\prod_{j=1}^n \gamma(g_j)^{p_{nj}}
\right).
\]
Since \(\gamma\) is a character and \(h_i=\sum_{j=1}^n p_{ij}g_j\), this becomes
\[
\left(
\gamma\!\left(\sum_{j=1}^n p_{1j}g_j\right),
\ \dots,\
\gamma\!\left(\sum_{j=1}^n p_{nj}g_j\right)
\right)
=
\bigl(\gamma(h_1),\dots,\gamma(h_n)\bigr)
=
\Phi^h_\gamma(1,\dots,1).
\]
Therefore
\[
\Psi_P\circ \Phi^g_\gamma=\Phi^h_\gamma\circ \Psi_P
\qquad \forall \gamma\in \widehat G,
\]
so \(\Psi_P\) is a conjugacy. More generally, for any \(c\in \mathbb T^n\), the map
\(z\mapsto c\cdot \Psi_P(z)\) is also a conjugacy.

\medskip

{\em (1) implies (2).} 
Assume that \(H:\mathbb T^n\to \mathbb T^n\) is a homeomorphism satisfying
\[
H\circ \Phi^g_\gamma=\Phi^h_\gamma\circ H
\qquad \forall \gamma\in \widehat G.
\]
Let
\[
c:=H(1,\dots,1)\in \mathbb T^n,
\qquad
K(z):=c^{-1}\cdot H(z).
\]
Then \(K\) is still a conjugacy, and \(K(1,\dots,1)=(1,\dots,1)\). Indeed,
\[
K\circ \Phi^g_\gamma
=
c^{-1}\cdot H\circ \Phi^g_\gamma
=
c^{-1}\cdot \Phi^h_\gamma\circ H
=
\Phi^h_\gamma\circ K
\qquad \forall \gamma\in \widehat G.
\]
We claim that \(K\) is a group homomorphism of $\mathbb{T}^n$. Define
\[
B(z,w):=K(zw)\,K(z)^{-1}K(w)^{-1},
\qquad z,w\in \mathbb T^n.
\]
Fix \(w\in \mathbb T^n\). Since \(\Phi^g_\gamma\) and \(\Phi^h_\gamma\) are translation actions and
\(K\) is a conjugacy, we have for every \(\gamma\in \widehat G\),
\[
K\bigl(\Phi^g_\gamma(u)\bigr)=\Phi^h_\gamma\bigl(K(u)\bigr)
=K(u)\cdot \Phi^h_\gamma(\mathbf{1})
\qquad \forall u\in \mathbb T^n.
\]
Hence
\begin{align*}
B\bigl(\Phi^g_\gamma(z),w\bigr)
&=
K\bigl(\Phi^g_\gamma(z)w\bigr)\,
K\bigl(\Phi^g_\gamma(z)\bigr)^{-1}K(w)^{-1}\\
&=
K(zw)\,\Phi^h_\gamma(\mathbf{1})\,
\cdot \bigl(K(z)\,\Phi^h_\gamma(\mathbf{1})\bigr)^{-1}\cdot K(w)^{-1}\\
&=
K(zw)K(z)^{-1}K(w)^{-1}\\
&=
B(z,w).
\end{align*}
Thus, for each fixed \(w\), the map \(z\mapsto B(z,w)\) is continuous and invariant
under the minimal action \((\Phi^g_\gamma)_{\gamma\in\widehat G}\). By minimality, it must be
constant. Evaluating at \(z_0=\mathbf{1}\), and using \(K(\mathbf{1})=\mathbf{1}\), we obtain
\[
B(z,w)=B(\mathbf{1},w)= K(w) K(\mathbf{1})^{-1}K(w) =\mathbf{1}.
\]
Therefore
\[
K(zw)=K(z)K(w)
\qquad \forall z,w\in \mathbb T^n.
\]
So \(K\) is a continuous group automorphism of \(\mathbb T^n\). It follows that
\[
K=\Psi_P
\]
for some \(P=(p_{ij})\in GL(n,\mathbb Z)\).

Finally, evaluating the conjugacy relation at \(z=(1,\dots,1)\), we get
\[
\Psi_P\bigl(\Phi^g_\gamma(1,\dots,1)\bigr)
=
K\bigl(\Phi^g_\gamma(1,\dots,1)\bigr)
=
\Phi^h_\gamma\bigl(K(1,\dots,1)\bigr)
=
\Phi^h_\gamma(1,\dots,1).
\]
That means
\[
\left(
\prod_{j=1}^n \gamma(g_j)^{p_{1j}},
\ \dots,\
\prod_{j=1}^n \gamma(g_j)^{p_{nj}}
\right)
=
\bigl(\gamma(h_1),\dots,\gamma(h_n)\bigr)
\qquad \forall \gamma\in \widehat G.
\]
Equivalently,
\[
\gamma\!\left(\sum_{j=1}^n p_{ij}g_j-h_i\right)=1
\qquad \forall \gamma\in \widehat G,\ \ 1\le i\le n.
\]
Since characters separate points in a locally compact abelian group, this implies
\[
h_i=\sum_{j=1}^n p_{ij}g_j,
\qquad 1\le i\le n.
\]
Thus \eqref{eq:Pg=h} is established, and
\[
H(z)=c\cdot K(z)=c\cdot \Psi_P(z).
\]

{\em Uniqueness of $P$}. 
Suppose \(P,Q\in M_n(\mathbb Z)\) both satisfy
\[
h=Pg=Qg. 
\]
Then
\(
(P-Q)g=0.
\)
Writing this row by row, each row of \(P-Q\) gives an integer relation among
\(g_1,\dots,g_n\). Since the action associated with \(g\) is minimal, we have
\[
\sum_{i=1}^n u_i g_i=0,\quad u_i\in\mathbb Z
\ \Longrightarrow\
u_1=\cdots=u_n=0.
\]
Hence every row of \(P-Q\) is zero, so \(P=Q\). This proves uniqueness.

 \subsection{Examples of actions under $\mathbb{R}^d$ and $\mathbb{Z}^d$}\label{sect:examples}
 Let us consider certain \(\mathbb{R}^d\)-flows and \(\mathbb{Z}^d\)-actions on the torus \(\mathbb{T}^n\). Throughout this discussion, we view \(\mathbb{T}^n\) as the additive group
\(
\mathbb{R}^n/\mathbb{Z}^n,
\)
so that its identity element is \(\mathbf{0}=(0,\dots,0)^T\).

To distinguish between the \(\mathbb{R}^d\)-flow and the \(\mathbb{Z}^d\)-action, we denote by \(H_{\mathbb{R}^d,\mathbb{T}^n}\) (respectively, \(H_{\mathbb{Z}^d,\mathbb{T}^n}\)) the closure of the orbit of \(\mathbf{0}\) under the \(\mathbb{R}^d\)-flow (respectively, the \(\mathbb{Z}^d\)-action) on \(\mathbb{T}^n\). In these two cases, Lemma~\ref{lem:Hperp} takes the form
\[
H_{\mathbb{Z}^d,\mathbb{T}^n}^{\perp}
=
\{\mathbf{u}\in\mathbb{Z}^n : L\mathbf{u}\in\mathbb{Z}^d\},
\qquad
H_{\mathbb{R}^d,\mathbb{T}^n}^{\perp}
=
\{\mathbf{u}\in\mathbb{Z}^n : L\mathbf{u}=\mathbf{0}\}.
\]

Recall that \(\mathbb{Z}^d=\widehat{\mathbb{T}^d}\) and \(\mathbb{R}^d=\widehat{\mathbb{R}^d}\). Thus the condition \(L\mathbf{u}\in\mathbb{Z}^d\) means precisely that \(L\mathbf{u}=0_{\mathbb{T}^d}\) in \(\mathbb{T}^d\), whereas \(L\mathbf{u}=\mathbf{0}\) is to be understood as \(L\mathbf{u}=\mathbf{0}_{\mathbb{R}^d}\) in \(\mathbb{R}^d\). In particular, both
\[
H_{\mathbb{Z}^d,\mathbb{T}^n}^{\perp}
\quad\text{and}\quad
H_{\mathbb{R}^d,\mathbb{T}^n}^{\perp}
\]
are \(\mathbb{Z}\)-submodules of \(\mathbb{Z}^n\).

By Theorem~\ref{thm:K} (Kronecker's theorem), the points of \(H_{\mathbb{Z}^d,\mathbb{T}^n}\) (respectively, \(H_{\mathbb{R}^d,\mathbb{T}^n}\)) are precisely those vectors whose coordinates can be simultaneously approximated by
\[
L_1\mathbf{u},\dots,L_n\mathbf{u},
\qquad
\mathbf{u}\in\mathbb{Z}^d
\quad
\text{(respectively, } \mathbf{u}\in\mathbb{R}^d\text{)}.
\]
Equivalently,
\[
H_{\mathbb{Z}^d,\mathbb{T}^n}
=
\left\{
\theta\in\mathbb{T}^n :
\forall \mathbf{u}\in H_{\mathbb{Z}^d,\mathbb{T}^n}^{\perp},
\ \sum_{j=1}^n u_j\theta_j \in \mathbb{Z}
\right\},
\]
and
\[
H_{\mathbb{R}^d,\mathbb{T}^n}
=
\left\{
\theta\in\mathbb{T}^n :
\forall \mathbf{u}\in H_{\mathbb{R}^d,\mathbb{T}^n}^{\perp},
\ \sum_{j=1}^n u_j\theta_j \in \mathbb{Z}
\right\}.
\]
We can identify $L_j$ with a column vector in $\mathbb{R}^d$. Then we can write $L_j \mathbf{u}$ as $L_j \cdot \mathbf{u}$ (inner product)
or $L^T_j \mathbf{u}$ (matrix product), where $^T$ denotes the transpose. 
If we use $L$ to denote the $d\times n$-matrix with columns $L_1, \cdots, L_n$. Then
the action can be additively represented by matrix:
$$
   \mathbf{x} \mapsto  L^T \mathbf{u} + \mathbf{x}= (L_1^T\mathbf{u}+x_1, \cdots, L_n^T\mathbf{u})^T +x_n \in \mathbb{T}^n
$$
for $\mathbf{x} =(x_1, \cdots, x_n) \in \mathbb{T}^n$.

    \begin{example}  Assume $d=n=1$ and $L=(\alpha)$. 
    \begin{enumerate}
    \item  If $\alpha$ is irrational, we have 
     $H_{\mathbb{Z}^1, \mathbb{T}^1}= H_{\mathbb{R}^1, \mathbb{T}^1}=\mathbb{T}^1$. 
    \item 
    If $\alpha=\frac{p}{q}$ is rational with $p$ and $q$ coprime, 
    then $H_{\mathbb{Z}^1, \mathbb{T}^1}= \mathbb{Z}/q\mathbb{Z}$ and $ H_{\mathbb{R}^1, \mathbb{T}^1}=\mathbb{T}^1$. 
    \end{enumerate}
       \end{example}
       
       This is because if $\alpha$ is irrational we have 
        $$ H_{\mathbb{Z}^1, \mathbb{T}^1}^\perp= H_{\mathbb{R}^1, \mathbb{T}^1}^\perp=\{0\}
        $$
        and if $\alpha = \frac{p}{q}$ with $p,q$ coprime we have
        $$ 
       H_{\mathbb{Z}^1, \mathbb{T}^1}^\perp=q\mathbb{Z}, \qquad H_{\mathbb{R}^1, \mathbb{T}^1}^\perp=\{0\}.
       $$

       \begin{example}  Assume $d=2$, $n=1$ and $L=(\alpha_1, \alpha_2)^T$.  Assume that 
       $\alpha_1=\frac{p_1}{q_1}$ with $p_1, q_1$ coprime,  and  $\alpha_2=\frac{p_2}{q_2}$ with $p_2, q_2$ coprime.
       Denote by $[q_1, q_2]$ the least common multiple of $q_1$ and $q_2$. 
       Then the orbit of $0\in \mathbb{T}$ under the $\mathbb{Z}^2$-action is equal to the cyclic group
        $$
           \frac{p_1}{q_1}\mathbb{Z} +   \frac{p_2}{q_2}\mathbb{Z}    \mod \mathbb{Z} = \mathbb{Z}/[q_1, q_2]\mathbb{Z}. 
       $$
           \end{example}
     
      Indeed, we have $H^\perp_{\mathbb{Z}^2, \mathbb{\mathbb{T}}} = \{n \in \mathbb{Z}:  n  (\frac{p_1}{q_1}, \frac{p_2}{q_2})^T \in \mathbb{Z}^2\} = [q_1, q_2]\mathbb{Z}$. So,
       $$
        H_{\mathbb{Z}^2, \mathbb{\mathbb{T}}}  =\{\theta\in \mathbb{T}: \forall n \in \mathbb{Z}, [q_1, q_2]n \theta \in \mathbb{Z}\} =  \mathbb{Z}/[q_1, q_2]\mathbb{Z}.
       $$

     \begin{example}  Assume $d=1$, $n=2$ and $L=(\alpha_1, \alpha_2)$.  Assume that 
       $\alpha_1=\frac{p_1}{q_1}$ with $p_1, q_1$ coprime,  and  $\alpha_2=\frac{p_2}{q_2}$ with $p_2, q_2$ coprime.
       Denote by $[q_1, q_2]$ the least common multiple of $q_1$ and $q_2$. 
       Then the orbit $H_{\mathbb{Z}^1, \mathbb{T}^2}$ of $\mathbf{0}\in \mathbb{T}^2$ under the $\mathbb{Z}$-action is equal to the product of  cyclic groups
         $$
          \left( \frac{p_1}{q_1},    \frac{p_2}{q_2}\right) \mathbb{Z}  \mod \mathbb{Z}^2=  
          (\mathbb{Z}/q_1\mathbb{Z}) \times (\mathbb{Z}/q_2\mathbb{Z}).
       $$
    \end{example}
    
     Indeed, we have $H^\perp_{\mathbb{Z}, \mathbb{T}^2} = \{\mathbf{u} \in \mathbb{Z}^2:  \frac{p_1}{q_1}u_1 + \frac{p_2}{q_2}u_2 \in \mathbb{Z}\} = [q_1, q_2]\mathbb{Z}$. So,
       $$
        H_{\mathbb{Z}^2, \mathbb{\mathbb{T}}}  =\{\theta\in \mathbb{T}: \forall n \in \mathbb{Z}, [q_1, q_2]n \theta \in \mathbb{Z}\} =  \mathbb{Z}/[q_1, q_2]\mathbb{Z}.
       $$

    In order to describe the next example, we need the notion of Pythagoras angle.
    We say that $\theta \in [0, 2\pi)$ is a {\em Pythagoras angle}  if $\sin \theta$ and $\cos \theta$ are rational numbers. In this case, we can assume      
    $$
       \sin \theta = \frac{a}{c}, \quad \cos \theta =\frac{b}{c} \quad (a,b,c\in \mathbb{Z}, a^2+b^2=c^2, a \ {\rm and} \  b \ {\rm are \ coprime}).
    $$
    
    \begin{example}  Assume $d=2$, $n=3$ and 
    $$
     L= \begin{pmatrix}
                1 & 0 & \cos \theta\\
                0 & 1 & \sin \theta
     \end{pmatrix}     \ \ \ {\rm with} \ \ \theta \in [0, 2\pi).
     $$
     \begin{enumerate}
     \item 
     If $\theta$ is not a Pythagoras angle, we have 
     $$
         H_{\mathbb{Z}^2, \mathbb{T}^3}=\{ (0,0)\} \times \mathbb{T}, \qquad H_{\mathbb{R}^2, \mathbb{T}^3}= \mathbb{T}^3.
     $$
     \item
      If $\theta$ is a Pythagoras angle with Pythagoras triple $(a, b, c)$, we have 
     $$
         H_{\mathbb{Z}^2, \mathbb{T}^3}= \{(0,0)\} \times \mathbb{Z}/c\mathbb{Z}, \qquad H_{\mathbb{R}^2, \mathbb{T}^3}= 
         \{\alpha\in \mathbb{T}^3: b \alpha_1 + a \alpha_2- c \alpha_3 \in \mathbb{Z}\}.
     $$
     \end{enumerate}

     \end{example}
     
     In the following, we give detailed analysis of both $\mathbb{Z}^2$-action and $\mathbb{R}^2$-action.
     
    {\bf  $\mathbb{Z}^2$-action}.
    That $\mathbf{u}\in H_{\mathbb{Z}^2, \mathbb{T}^3}^\perp$ means that for some integers $m_1$ and $m_2$ we have
    $$
          u_1 + u_3 \cos \theta = m_1, \qquad  u_2 + u_3 \sin \theta =m_2.
    $$
    Clearly 
    $\mathbb{Z}^2 \times \{0\}\subset H_{\mathbb{Z}^2, \mathbb{T}^3}^\perp$. 
     Let us first find the condition for $H_{\mathbb{Z}^2, \mathbb{T}^3}^\perp= \mathbb{Z}^2 \times \{0\}$. 
    Now assume $u_3\not=0$. It is clear that if one of $\sin \theta$ and $\cos \theta$ is irrational, the above system has no solution. 
    We conclude that $H_{\mathbb{Z}^2, \mathbb{T}^3}^\perp= \mathbb{Z}^2 \times \{0\}$ when $\theta$ is not a Pythagoras angle.
    Then it is clear that
    $$
          H_{\mathbb{Z}^2, \mathbb{T}^3}= \{(0,0)\} \times \mathbb{T}. 
    $$

    Now assume $\sin \theta = \frac{a}{c}$ and $\cos \theta =\frac{b}{c}$, where $a,b,c$ are rational integers verifying $a^2+b^2=c^2$,
    with $a$ and $b$ coprime. The above system becomes 
        $$
          u_1 + u_3 \frac{b}{c} = m_1, \qquad  u_2 + u_3 \frac{a}{c} =m_2.
    $$
    Its solvability implies that $c|u_3$. So, write $u_3=cu_3'$. We rewrite the above system  as follows
      $$
          u_1 + b u_3'  = m_1, \qquad  u_2 + a u_3' =m_2.
    $$
    This implies $b|(m_1-u_1)$ and $a|(m_2-u_2)$. Write $m_1= u_1 + bk_1$ and $m_2=u_2+a k_2$. The system is now transformed into
    $ u_3' = k_1=k_2$. Therefore any integral triple  $(u_1, u_2, u_3')$ is a solution of the last system  if we take $m_1 = u_1+ b u_3'$ and $m_2=u_2+au_3'$.
    Thus, in this case, we get $H_{\mathbb{Z}^2, \mathbb{T}^3}^\perp= \mathbb{Z}^2 \times c \mathbb{Z}$. It follows that 
    $$
     H_{\mathbb{Z}^2, \mathbb{T}^3}=  \{(0,0)\}\times (\mathbb{Z}/c\mathbb{Z}).
    $$
    Of course, by iteration, we get immediately the orbit of $\mathbf{0}$, because
    $$
         L^T\mathbf{x} = (x_1, x_2, x_1\cos \theta +x_2\sin \theta) = (0,0, x_1\cos \theta +x_2\sin \theta) \ \ \mod \mathbb{Z}^3.
    $$

     {\bf  $\mathbb{R}^2$-action}. It is easier.
    That $\mathbf{u}\in H_{\mathbb{R}^2, \mathbb{T}^3}^\perp$ means that for some integers $m_1$ and $m_2$ we have
    $$
          u_1 + u_3 \cos \theta = 0, \qquad  u_2 + u_3 \sin \theta =0.
    $$
    If $\theta$ is not a Pythagoras angle, it is obvious that $\mathbf{u}\in H_{\mathbb{R}^2, \mathbb{T}^3}^\perp = \{(0,0,0)\}$ so that 
    $\mathbf{u}\in H_{\mathbb{R}^2, \mathbb{T}^3}= \mathbb{Z}^3$. 
    Suppose that  $\theta$ is not a Pythagoras angle with its triple $(a, b, c)$. The above system becomes 
    the following indeterminate system
    $$
          cu_1+ b u_3 =0, \qquad cu_2 + a u_3=0.
    $$
    First observe that $b|u_1, a|u_2, c|u_3$.  Write $u_1=bu_1', u_2=a u_2', u_3=cu_3'$. Then the above system becomes 
    $$
        u_1'+u_3' =0. \quad u_2'+ u_3'=0, \ \ i.e. \ \ u_1=u_2=-u_3.
    $$
    Thus we get $\mathbf{u}\in H_{\mathbb{R}^2, \mathbb{T}^3}^\perp = (b, a, -c)\mathbb{Z}$.
    Therefore $\alpha \in H_{\mathbb{R}^2, \mathbb{T}}$ means $n(b\alpha_1+ a\alpha_2 -c\alpha_3) \in \mathbb{Z}$ for all
    integer $n$, equivalently  $b\alpha_1+ a\alpha_2 -c\alpha_3=0 \mod \mathbb{Z}$. 
    \medskip
    
    Now let us state of condition of minimality or equivalently unique ergodicity.

   \begin{example}  Given $\alpha =(\alpha_1, \dots, \alpha_n)\in \mathbb{R}^n$. The translation $T_{\alpha}: \mathbb{T}^n \to \mathbb{T}^n$ defined by 
   $$
         T_\alpha(z) = z+ \alpha = (z_1 +\alpha_1, \cdots, z_n +\alpha_n)
   $$
   is the generator of the $\mathbb{Z}$-action defined by $\Phi_n(z) = z + n \alpha$ with $n \in \mathbb{Z}$. By Theorem \ref{thm:W}, this action is uniquely ergodic if and only if 
   $\alpha_1, \dots, \alpha_n$ is $\mathbb{Z}$-independent modulo $\mathbb{Z}$. 
   
   In the usual terms, this means that $1, \alpha_1, \dots, \alpha_n$ are $\mathbb{Q}$-independent.
   This result  is well known. On the other hand, by Theorem \ref{thm:W}, the $\mathbb{R}$-flow defined by $\Phi_t(z) = z+t\alpha$ with $t\in \mathbb{R}$ is uniquely ergodic if and only if $\alpha_1, \dots, \alpha_n$ are $\mathbb{Q}$-independent. This fact is useful for numerical computation \cite{JZ2014}.
   \medskip
   
   Remark that the condition of unique ergodicity for $\mathbb{Z}$-action is stronger than that for $\mathbb{R}$-flow when $n \ge 2$.
   When $n=1$, the irrationality ensures the ergodicity of both  $\mathbb{Z}$-action and  $\mathbb{R}$-flow. 
   When $n=2$, $\alpha=(1/2, \sqrt{2})$ produces a unique ergodic $\mathbb{R}$-flow, but the corresponding $\mathbb{Z}$-action is not unique ergodic.
   \end{example}

    \begin{example}  Given two vectors $\alpha =(\alpha_1, \dots, \alpha_n)$ and $\beta =(\beta_1, \dots, \beta_n)$ in $\mathbb{R}^n$. Let us consider the $\mathbb{R}^2$-flow on $ \mathbb{T}^n$ defined by 
   $$
         \Phi_{t_1, t_2}(z) = z+ t_1\alpha + t_2 \beta = (z_1 + t_1\alpha_1+ t_2 \beta_1, \cdots, z_n +t_1\alpha_n+t_2\beta_n).
   $$
   Our matrix is equal to
   $$
     P=(p_1, \cdots, p_n) = \begin{pmatrix} 
            \alpha_1 & \alpha_2 & \cdots & \alpha_n\\
            \beta_1 & \beta_2 & \cdots & \beta_n
     \end{pmatrix}.
  $$
   By Theorem \ref{thm:W}, a necessary and sufficient condition for the  $\mathbb{R}^2$-flow on $ \mathbb{T}^n$ to be uniquely ergodic is that the following system has no non-trivial solution
   $(k_1, \cdots, k_n) \in \mathbb{Z}^n$
      \begin{equation}\label{eq:UE-R2Flow}
           \begin{cases}
                 \alpha_1 k_1 + \cdots +  \alpha_n k_n=0& \\
                 \beta_1 k_1 + \cdots +  \beta_n k_n =0. &
     \end{cases}  
     \end{equation}
      Again by Theorem \ref{thm:W}, a necessary and sufficient condition for the  $\mathbb{Z}^2$-action on $ \mathbb{T}^n$ to be uniquely ergodic is that the following system has no non-trivial solution
   $(k_1, \cdots, k_n) \in \mathbb{Z}^n$
 \begin{equation}\label{eq:UE-Z2Flow}
           \begin{cases}
                 \alpha_1 k_1 + \cdots +  \alpha_n k_n=0 \mod 1\\
                 \beta_1 k_1 + \cdots +  \beta_n k_n =0 \mod 1   .
           \end{cases}
       \end{equation}  
      \end{example}

     {\em Remark.}     
     The above discussion concerning  with $d=2$ is valid for arbitrary time dimension $d$. The unique ergodicity for  $\mathbb{R}^d$-flows and $\mathbb{Z}^d$-actions on $\mathbb{T}^n$
    will be expressed by a system of $d$ equations. 
      
     We can even go further.  Given infinitely many linear forms $L_1, \cdots, L_n, \cdots$. We can consider the $\mathbb{R}^d$-flows (resp. $\mathbb{Z}^d$-actions) on the infinite-dimensional torus $\mathbb{T}^\mathbb{N}$. Then 
     the condition of unique ergodicity  is expressed 
     by \eqref{eq:UE-R2Flow} (resp. \eqref{eq:UE-Z2Flow}), but with an arbitrary integer $n$.

     \begin{example} 
   Assume that $\alpha\in \mathbb{R}$ is transcendental. Then the $\mathbb{R}$-flow on $\mathbb{T}^\mathbb{N}$ defined by
$$
         \Phi_t(z) = (z_0+ t, z_1 + \alpha t, z_2 +\alpha^2 t, \cdots, z_n + \alpha^n t, \cdots)
$$   
is uniquely ergodic.  But  it is not unique ergodic if $\alpha$ is algebraic. Indeed, 
the orbit of the neutral point of  $\mathbb{T}^\mathbb{N}$ consists of the points $$(t, \alpha t, \cdots, \alpha^n t, \cdots).$$
     If $\alpha$ is  algebraic of degree $m$, then the coordinates $\alpha^n t$ with $n \ge m$ are functions of the first 
     $m$ coordinates. So, the closed orbit of the neutral point is an $m$-dimensional topological torus. 

\end{example}

 \section{Applications to harmonic analysis on LCA groups}\label{sect:group}
 
 Theorem~\ref{thm:group} is one of the fundamental results of this paper and is derived from the unique ergodic decomposition theorem, Theorem~\ref{thm:UE}. Both Theorem~\ref{thm:Wiener} and Theorem~\ref{thm:BM} then follow as consequences of Theorem~\ref{thm:group}.

 .
\subsection{Proof of Theorem \ref{thm:group}}

We present two proofs. One is based on Theorem \ref{thm:UE} and the other is based on  Pontryagin duality.

{\em First proof.} Fix $g\in G\setminus\{0\}$. Consider the 
$\widehat{G}$-action on $\mathbb{T}$ defined by 
$\Phi_\gamma(z)= z\gamma(g)$. Let us consider  the closed orbit of $1$:
$$
H:= \overline{ \{\gamma(g) \in \mathbb{T}: \gamma \in \widehat{G}\}}.
$$
It is the whole $\mathbb{T}$ or a cyclic subgroup of $\mathbb{T}$. But it is not the trivial group $\{1\}$
because $g\not =0$. 
By Theorem \ref{thm:W}, the action restricted on $H$ is uniquely ergodic. So, by Theorem \ref{thm:UE}, for any continuous function
$\varphi$ on $H$ and for any F{\o}lner  sequence $(F_n)$ of $\widehat{G}$, we have
$$
\forall z\in H, \quad \lim_{n\to \infty}\frac{1}{|F_n|}
\int_{F_n} \varphi(\Phi_\gamma(z)) = \int_H \varphi(w) d\mathbf{m}_H(w)
$$
where $\mathbf{m}_H$ is the Haar measure of $H$.
Take $\varphi(w) =w$,  which is a non-trivial character of $H$, so that
$\int_H \varphi(w) d\mathbf{m}_H(w)=0$. On the other hand, 
$\varphi(\Phi_\gamma(1))= \gamma(g)$. Thus, the above equality applied to $\varphi(w)=w$ and $z=1$ gives the desired result
\eqref{eq:W0}.

{\em Second proof (by Hanfeng Li).} For any $g \in G\setminus \{0\}$, by Pontryagin duality one can find some $\gamma'\in \widehat{G}$ such that $\gamma'(g)\neq 1$. For each $n$, one has
$$ \gamma'(g)\int_{F_n} \gamma(g)\, d\gamma=\int_{F_n}(\gamma+\gamma')(g)\, d\gamma=\int_{F_n+\gamma'} \gamma(g)\, d\gamma.$$
Thus
\begin{align*}
\left|(\gamma'(g)-1)\frac{1}{|F_n|}\int_{F_n}\gamma(g)\, d\gamma \right|&=\left|\frac{1}{|F_n|}\left(\int_{F_n+\gamma'} \gamma(g)\, d\gamma-\int_{F_n} \gamma(g)\, d\gamma\right)\right|\\
&\le \frac{1}{|F_n|}|F_n\Delta (F_n+\gamma')|\to 0,
\end{align*}
as $n\to \infty$. Therefore \eqref{eq:W0} holds.

\subsection{Proof of Theorem \ref{thm:Wiener}}
The essential new input is Theorem~\ref{thm:group}; apart from this, the proof is the same as in the classical setting.
Fix \(x\in G\), and define
\[
g_n(y):=\frac{1}{|F_n|}\int_{F_n}\gamma(x-y)\,d\gamma
      =\frac{1}{|F_n|}\int_{F_n}\gamma(x)\overline{\gamma(y)}\,d\gamma.
\]
By Theorem~\ref{thm:group}, \(g_n(y)\to 0\) for every \(y\neq x\), while \(g_n(x)\to 1\). On the other hand, \(|g_n(y)|\le 1\) for all \(y\in G\) and all \(n\). Hence, by the dominated convergence theorem,
\[
\mu(\{x\})
=
\lim_{n\to\infty}\int_G g_n(y)\,d\mu(y)
=
\lim_{n\to\infty}\frac{1}{|F_n|}\int_G
\left(\int_{F_n}\gamma(x)\overline{\gamma(y)}\,d\gamma\right)d\mu(y).
\]
Applying Fubini's theorem then yields~(1).

To prove~(2), we apply~(1) to the measure \(\mu*\widetilde{\mu}\) at the point \(x=0\), where \(\widetilde{\mu}\) is defined by
\[
\widetilde{\mu}(A):=\overline{\mu(-A)}.
\]
Equivalently, \(\widetilde{\mu}\) is characterized by
\(
\widehat{\widetilde{\mu}}(\gamma)=\overline{\widehat{\mu}(\gamma)}.
\)
We then obtain
\[
(\mu*\widetilde{\mu})(\{0\})
=
\lim_{n\to\infty}\frac{1}{|F_n|}\int_{F_n}|\widehat{\mu}(\gamma)|^2\,d\gamma.
\]
However, by decomposing $\mu$ into its discrete part and continuous part, one gets
\[
(\mu*\widetilde{\mu})(\{0\})
=
\sum_{x\in G}|\mu(\{x\})|^2,
\]
We have thus proved (2).

\subsection{Proof of Theorem \ref{thm:BM}}

A fundamental result concerning LCA groups is the Bohr orthogonality of group characters. It is a direct consequence of \ref{thm:group} . Because of its importance for us, we state it as the following lemma.

\begin{lemma}\label{lem:orthog-gamma}
     Let $\widehat{G}$ be the Pontryajin dual group of a LCA group $G$. We have $\mathcal{M}(\gamma' \cdot \overline{\gamma''})=0$ for any two distinct characters
$\gamma', \gamma''\in \widehat{G}$. 
In particular, $\mathcal{M}(\gamma)=0$
for all non-trivial character $\gamma$.
    \end{lemma}
    
    To derive this lemma from Theorem~\ref{thm:group}, we regard \(G\) as the dual of \(\widehat{G}\), since Theorem~\ref{thm:group} is formulated in terms of averages over the dual group.
    
    We now prove Theorem~\ref{thm:BM} as consequence of Lemma \ref{lem:orthog-gamma} and Bohr's theorem stating that almost periodic functions are uniformly limits of trigonometric polynomials.   
Let us first consider a trigonometric polynomial on \(G\) of the form
\[
P(x)=a_0+\sum_{\gamma} a_\gamma \gamma(x)
\]
where the sum ranges over a finite set of nontrivial characters of \(G\). By the second assertion of Lemma~\ref{lem:orthog-gamma}, we have
\[
\mathcal{M}(P)=a_0.
\]
Now let \((P_m)\) be a sequence of trigonometric polynomials converging uniformly to \(f\), and let \(a_m\) denote the constant term of \(P_m\). By the above argument, \(a_m=\mathcal{M}(P_m)\) for every \(m\). Moreover,
\[
|a_m|=|\mathcal{M}(P_m)|\le \|P_m\|_\infty,
\]
so the sequence \((a_m)\) is bounded.

Fix \(\varepsilon>0\). Since \(P_m\to f\) uniformly, there exists \(M>0\) such that
\begin{equation}\label{eq:UC}
|f(x)-P_m(x)|\le \varepsilon
\qquad
\text{for all }x\in G\text{ and all }m\ge M.
\end{equation}
It follows that, for every \(m\ge M\),
\begin{equation}\label{eq:BM1}
\limsup_{n\to\infty}
\left|
\frac{1}{|F_n|}\int_{F_n} f(x)\,dx-a_m
\right|
\le
\limsup_{n\to\infty}
\frac{1}{|F_n|}\int_{F_n}|f(x)-P_m(x)|\,dx
\le \varepsilon.
\end{equation}
Let \(\ell\) be any limit point of \((a_m)\). Passing to a subsequence \(a_{m_k}\to \ell\) in \eqref{eq:BM1}, we obtain
\[
\lim_{n\to\infty}\frac{1}{|F_n|}\int_{F_n} f(x)\,dx=\ell.
\]
We have thus proved that  the limit defining \(\mathcal{M}(f)\) exists and  \(\mathcal{M}(f) =\ell\). Since $\ell$  depends only on the constant terms \(a_m\), it follows that \(\mathcal{M}(f)\) is independent of the choice of the Følner sequence \((F_n)\). So, $(a_m)$ has a unique limit point. This shows that the sequence \((a_m)\) actually converges, with
\[
\lim_{m\to\infty} a_m=\mathcal{M}(f).
\]

Finally, observe that the polynomial \(P(x)\) and its translate \(P(x+g)\) have the same constant term. Hence the previous argument remains valid with \(F_n\) replaced by \(F_n-g\). More precisely, for every \(g\in G\) and every \(m\ge M\),
\begin{equation*}\label{eq:BM2}
\limsup_{n\to\infty}
\left|
\frac{1}{|F_n|}\int_{F_n-g} f(x)\,dx-a_m
\right|
\le
\limsup_{n\to\infty}
\frac{1}{|F_n|}\int_{F_n}|f(x+g)-P_m(x+g)|\,dx
\le \varepsilon.
\end{equation*}
where the estimate \eqref{eq:UC} is used. Thus we have proved that 
 the convergence
\[
\frac{1}{|F_n|}\int_{F_n-g} f(x)\,dx \longrightarrow \mathcal{M}(f)
\]
is uniform in \(g\in G\).

\section{Application to  numerical solution of eigenproblem of Schr\"{o}dinger equation}

In this last section, we develop another example in the spirit of Example 4 in Section \ref{sect:examples}. 
It  confirms that the numerical projection indicator method used by physicists is effective to solve the eigenproblem of Schr\"{o}dinger equation.

We follow \cite{GXYY2023} to present this numerical method.
In a variety of fields such as materials
science, condensed-matter physics, optics, and electronics, 
moir\'e systems arise. 
A moir\'e system is a system that involves two or more periodic
structures with different lattice  orientation, which
interact with each other to form a spatial moiré pattern. 
For example, Moir\'e patterns are  crucial indications in all  physical areas
related to wave propagation, such as Bose-Einstein condensates and optics. Moir\'e patterns  enable  the possibility to
explore the phenomena  of transition
from aperiodic (incommensurate) to periodic (commensurate)
geometries, occurring at specific values of the rotation angle
in contrast to aperiodic quasicrystal systems.
The mathematical model in these fields is a
Schrödinger-like equation 
$$
   i\frac{\partial \psi}{\partial t} = - \frac{1}{2} +  V(\mathbf{r}) \psi, \quad  \text{with} \ \ \ 
   V(\mathbf{r}) =  \frac{E_0}{1+ I(\mathbf{r})} 
$$
where $\mathbf{r} =(x, y) \in \mathbb{R}^2$ and $$
I(\mathbf{r}) = |p_1 v(\mathbf{r}) +p_2 v(R_\theta(\mathbf{r}))|^2
$$ is the intensity of the moir\'e lattice induced by two ordinarily polarized mutually coherent periodic sublattices, 
$v(\mathbf{r})$ and  $v(R_\theta(\mathbf{r}))$. Here $R_\theta$ is a rotation of angle $\theta$ of the plane.  
The constants $p_1, p_2$ represent the amplitudes of the first and the secon sublattices and the latices ratio is defined by the quotient $p_1/p_2$, which will be assumed to be $1$.  
The potential $V(\mathbf{r})$ describes the optical response of the photorefractive cristal with strength $E_0$. Physicists observe that the moir\'e lattices of such two square lattices are periodic when 
$\theta$ is a Pythagoras angle (i.e. $\tan \theta$ is rational) and aperiodic otherwise.  

Physicists (cf. \cite{GXYY2023}) developed a numerical method,  called projective indicator method (PI method for short),  to solve the eigenproblem of the Schr\"odinger equation
\begin{equation}\label{eq:schrod}
     E \psi = - \frac{1}{2} \Delta \psi +  V(\mathbf{r}) \psi
\end{equation}
by combining the projection method \cite{JZ2014} and the indicator projection with the plane-wave discretization (Fourier method). 
For example, we can choose the projection matrix
$$
   \mathbf{P}= 2 \begin{pmatrix}
\cos\gamma & -\sin\gamma & \cos\gamma & \sin\gamma\\
\sin\gamma& \cos\gamma & -\sin\gamma & \cos\gamma
\end{pmatrix} \quad \text{with} \ \ \ \gamma = \frac{\theta}{2}.
$$
Notice that $   \mathbf{P} = 2(R_\gamma, R_{-\gamma})$. By the projection matrix $\mathbf{P}$, we transit 
from $2D$ space to $4D$ space, through $\mathbf{q}=\mathbf{P}^T\mathbf{r}$, with $\mathbf{q}=(q_1, q_2, q_3, q_4)\in \mathbb{R}^4$. 
Assume $\psi (\mathbf{r})= \phi\circ \mathbf{q}$ for some periodic function of $4$ variables with period $2\pi$ for each variable, where $\mathbf{q}= \mathbf{P}^T \mathbf{r}$ (i.e. $\psi$ is assumed quasi-periodic).  
Then the equation \eqref{eq:schrod} becomes an  eigenproblem for periodic function
\begin{equation}\label{eq:schrod2}
     E \phi = - \frac{1}{2} \sum_{1\le i, j\le 4} \frac{\partial^2\phi}{\partial q_i \partial q_j}
     \left(\frac{\partial q_i}{\partial x} \frac{\partial q_j}{\partial x}+ \frac{\partial q_i}{\partial y} \frac{\partial q_j}{\partial y} \right) +  U(\mathbf{q}) \phi
\end{equation}
under the assumption that $V$ is quasi-periodic, i.e., $V(\mathbf{r}) =U\circ \mathbf{q}$ with $\mathbf{q}= \mathbf{P}^T \mathbf{r}$ for some $U$.  

Fix a large integer $N\ge 1$ and let $\Omega_N = [-N, N]^4 \cap \mathbb{Z}^4$. 
To numerically solve \eqref{eq:schrod2}, one expands the numerical unknown function $\phi_N$  using the plane-wave expansion (partial Fourier series)
\begin{equation}\label{eq:Fourier}
     \phi_N(\mathbf{q}) = \sum_{\mathbf{k} \in \Omega_N} \widehat{\phi}_N(\mathbf{k}) e^{i \mathbf{k} \cdot \mathbf{q}}, \qquad \mathbf{q} \in \left(\mathbb{R}/2\pi \mathbb{Z}\right)^4.
\end{equation}  
Taking Fourier transform on the both sides of \eqref{eq:schrod2}, the Fourier coefficients of the unknown function $\phi_N$ satisfies the system of  equations: 
\begin{equation}\label{eq:matrix}
  E \widehat{\phi_N}(\mathbf{k}) = \frac{1}{2} \sum_{i=1}^2 \sum_{1\le j, \ell\le 4} \widehat{\phi_N}(\mathbf{k})k_jk_\ell \frac{\partial q_j}{\partial r_i}  \frac{\partial q_\ell}{\partial r_i}
  +\widehat{U\psi_N}(\mathbf{k}).
\end{equation}
Effective numerical methods are needed to solve this equation \eqref{eq:matrix} (cf.   \cite{GXYY2023}, \cite{JLM3Z2025}, \cite{JZ2014}, \cite{JZ2018}). When the coefficients $\widehat{\phi}_N(\mathbf{k}) $  are computed, we get an approximative solution $\phi_N$ of the equation 
\eqref{eq:schrod2} by the formula \eqref{eq:Fourier}  then an approximative solution $\psi_N$ of the equation of 
\eqref{eq:schrod}  through $\phi_N (\mathbf{q})= \psi_N(\mathbf{r})$, $\mathbf{q} =\mathbf{P}^T\mathbf{r}$.

But $\mathbf{P}^T$ is not invertible. We can not say that $\psi_N(\mathbf{r})=\phi_N((\mathbf{P}^{-1} \mathbf{q}))$. 
However, there is a correspondence between $\psi$'s and $\phi$'s (cf. \cite{FJZ2025}, \cite{JZ2014}) under the condition that the $\mathbb{Q}$-rank of $\mathbf{P}$ is equal to two, 
or equivalently, the $\mathbb{R}^2$-flow on $\mathbb{T}^4$ is uniquely ergodic. We are going to show that 
this flow is uniquely ergodic if and only if $\tan \gamma$ is irrational (or the angle $\gamma$ is not Pythagoras).

The factor $2$ in the projection matrix $\mathbf{P}$ has no impact on  the dynamical property that we are interested in. So, we consider
\[
L=
\begin{pmatrix}
\cos\gamma & -\sin\gamma  & \cos\gamma  & \sin\gamma \\
\sin\gamma  & \cos\gamma  & -\sin\gamma  & \cos\gamma 
\end{pmatrix},
\qquad \gamma \in[0,2\pi),
\]
and consider the \( \mathbb R^2\)-flow on \(\mathbb T^4=\mathbb R^4/\mathbb Z^4\)
defined, in additive notation, by
\[
\Phi_t(x)=x+L^T t \pmod{\mathbb Z^4},
\qquad t\in\mathbb R^2,\ x\in\mathbb T^4.
\]
Let
\[
H:=H_{\mathbb R^2,\mathbb T^4}:=\overline{{\rm Orb}(0)}
\]
be the closure of the orbit of \(0\in\mathbb T^4\). We have
\[
H^\perp=\{u\in\mathbb Z^4:\ Lu=0\}.
\]
According to Theorem \ref{thm:W} (d), the flow is unique ergodic if and only if $H^\perp=\{0\}$.

\begin{thm}
The flow is  uniquely ergodic if and only if
\[
\tan\gamma \notin\mathbb Q\cup\{\infty\}.
\]

\end{thm}

\begin{proof}
We work throughout in additive notation on \(\mathbb T^4=\mathbb R^4/\mathbb Z^4\).

\noindent
\emph{Step 1: System \(Lu=0\).}
Write
\(
u=(u_1,u_2,u_3,u_4)\in\mathbb Z^4.
\)
The system \(Lu=0\) is equivalent to
\begin{equation}\label{eq:S1}
\begin{cases}
\cos\gamma\,(u_1+u_3)+\sin\gamma\,(-u_2+u_4)=0,\\[1mm]
\sin\gamma\,(u_1-u_3)+\cos\gamma\,(u_2+u_4)=0.
\end{cases}
\end{equation}
Introduce
\[
A=u_1+u_3,\qquad B=-u_2+u_4,\qquad
C=u_1-u_3,\qquad D=u_2+u_4.
\]
Then \eqref{eq:S1}  becomes
\begin{equation}\label{eq:S2}
\cos\gamma\,A+\sin\gamma\,B=0,\qquad
\sin\gamma\,C+\cos\gamma\,D=0.
\end{equation}

\medskip

\noindent
\emph{Step 2: $H^\perp =\{0\}$ in the irrational-slope case.}
Assume
\(
\tan\theta\notin\mathbb Q\cup\{\infty\}.
\)
Then \(\sin\gamma\neq 0\) and \(\cos\gamma\neq 0\), and both \(\tan\gamma\) and
\(\cot\gamma\) are irrational. From \eqref{eq:S2} we get
\[
A=-\tan\gamma\, B,\qquad C=-\cot\gamma\, D.
\]
Since \(A,B,C,D\in\mathbb Z\), the irrationality forces
\(
A=B=C=D=0.
\)
Therefore
\[
u_1+u_3=u_1-u_3=0,\qquad -u_2+u_4=u_2+u_4=0,
\]
hence \(u_1=u_2=u_3=u_4=0\). Thus
\(
H^\perp=\{0\}
\) and the flow is uniquely ergodic.

\medskip

\noindent
\emph{Step 3: $H^\perp \not=\{0\}$ in the rational-slope case.}
Assume now
\(
\tan\theta\in\mathbb Q\cup\{\infty\}.
\)
Choose coprime integers \(a,b\), not both zero, and \(\lambda\neq 0\) such that
\[
\sin\gamma=\lambda a,\qquad \cos\gamma=\lambda b.
\]
Then \eqref{eq:S2} becomes
\begin{equation}\label{eq:S3}
bA+aB=0,\qquad aC+bD=0.
\end{equation}
Since \((a,b)=1\), the integer solutions of \eqref{eq:S3} are
\[
A=-at,\qquad B=bt,\qquad C=-bu,\qquad D=au
\]
for some \(t,u\in\mathbb Z\). Recovering \(u_1,u_2,u_3,u_4\) from
\[
A=u_1+u_3,\quad C=u_1-u_3,\quad B=-u_2+u_4,\quad D=u_2+u_4,
\]
we obtain
\begin{equation}\label{eq:S4}
u_1=\frac{-at-bu}{2},\qquad
u_2=\frac{au-bt}{2},\qquad
u_3=\frac{-at+bu}{2},\qquad
u_4=\frac{au+bt}{2}.
\end{equation}
Therefore \(u\in\mathbb Z^4\) if and only if
\[
at+bu\equiv 0\pmod 2,
\qquad
au+bt\equiv 0\pmod 2,
\]
and this proves the general formula
\[
H^\perp
=
\left\{
\left(
\frac{-at-bu}{2},
\frac{au-bt}{2},
\frac{-at+bu}{2},
\frac{au+bt}{2}
\right)
:\ t,u\in\mathbb Z,\ 
at+bu\equiv au+bt\equiv 0\pmod 2
\right\}.
\]

\medskip

\emph{Case I: \(a\) and \(b\) have opposite parity.}
Then the two congruence conditions
\[
at+bu\equiv 0\pmod 2,\qquad au+bt\equiv 0\pmod 2
\]
force \(t\) and \(u\) to be even. Write
\[
t=2m,\qquad u=2n.
\]
Substituting into \eqref{eq:S4} gives
\[
u=m(-a,-b,-a,b)+n(-b,a,b,a).
\]
Hence
\[
H^\perp
=
\mathbb Z(-a,-b,-a,b)\oplus \mathbb Z(-b,a,b,a),
\]
which is not trivial. So, the flow is not uniquely ergodic.

\smallskip

\emph{Case II: \(a\) and \(b\) are both odd.}
Then
\[
at+bu\equiv t+u\pmod 2,\qquad au+bt\equiv t+u\pmod 2,
\]
so the parity condition is exactly
\[
t\equiv u\pmod 2.
\]
Hence there exist \(m,n\in\mathbb Z\) such that
\[
t=m+n,\qquad u=m-n.
\]
Substituting into \eqref{eq:S4}, we obtain
\[
u=m\,w_1+n\,w_2,
\]
where
\[
w_1=
\left(
-\frac{a+b}{2},
\frac{a-b}{2},
\frac{b-a}{2},
\frac{a+b}{2}
\right),
\qquad
w_2=
\left(
\frac{b-a}{2},
-\frac{a+b}{2},
-\frac{a+b}{2},
\frac{b-a}{2}
\right).
\]
Therefore
\[
H^\perp=\mathbb Z w_1\oplus \mathbb Z w_2,
\]
which is not trivial. So, the flow is not uniquely ergodic.
\end{proof}

\end{document}